\documentclass[10pt]{amsart}

\usepackage[usenames,dvipsnames]{xcolor}
\usepackage[bookmarks,colorlinks=true,citecolor=OliveGreen,linkcolor=RoyalBlue]{hyperref}
\usepackage{amssymb,amsthm}
\usepackage{mathtools}
\usepackage{mathabx}
\usepackage{xparse}
\usepackage{enumitem}
\usepackage{thmtools}
\usepackage{cleveref}		% Use \cref without environment name
\usepackage{bm}
\usepackage{tikz, soul}
\usetikzlibrary{calc,arrows}
\usepackage{cancel}

\numberwithin{equation}{section}

\newcommand{\seq}[1]{{\left\langle{#1}\right\rangle}}

\newcommand{\smallseq}[1]{\langle{#1}\rangle}
\newcommand{\tth}{{}^{\textup{th}}}
\newcommand{\conc}{\hat{\,\,}}
\newcommand{\rest}[1]{\! \upharpoonright\!{#1}}

\newcommand{\andd}{\,\,\,\&\,\,\,}

\newcommand{\Iff}{\,\,\Longleftrightarrow\,\,}

\newcommand{\converge}{\!\!\downarrow}

\newcommand{\Tur}{\textup{\scriptsize T}}

\newcommand{\Int}{\mathbb Z}

\newcommand{\Nat}{\mathbb N}

\newcommand{\w}{\omega}
\newcommand{\s}{\sigma}

\renewcommand{\phi}{\varphi}
\renewcommand{\epsilon}{\varepsilon}

\renewcommand{\le}{\leqslant}
\renewcommand{\ge}{\geqslant}

\renewcommand{\preceq}{\preccurlyeq}
\renewcommand{\succeq}{\succcurlyeq}

\newcommand{\tleq}{\trianglelefteq}

\newcommand{\tle}{\triangleleft}

\newcommand{\treq}[2]{\trianglerighteq^{#1}_{#2}}

\newcommand{\Club}{\clubsuit}

\newcommand{\wock}{\w_1^{\textup{ck}}}

\newcommand{\bSigma}{\boldsymbol{\Sigma}}
\newcommand{\bDelta}{\boldsymbol{\Delta}}
\newcommand{\bGamma}{\boldsymbol{\Gamma}}

\newcommand{\bPi}{\boldsymbol{\Pi}}

\newcommand{\ATR}{\mathsf{ATR}}       % Arithmetic transfinite recursion
\newcommand{\PiCA}{\Pi^1_1\text{-}\mathsf{CA}}    % Pi-1-1 comprehension
\newcommand{\IND}{\mathsf{IND}}

\renewcommand{\*}[1]{\mathbf{#1}}

\renewcommand{\b}[1]{\overline{#1}}

\newcommand{\Baire}{\mathcal{N}}

\newcommand{\emptystring}{\seq{}}

\newcommand{\RefClub}{{\hyperref[TSP:Club]{\textup{}($\Club$)}}}

\newcommand{\ClassName}[1]{\textup{#1}}
\newcommand{\SU}[2]{\ClassName{SU}_{#1}(#2)}
\newcommand{\ISU}[2]{\ClassName{ISU}^{(#1)}_{#2}}

\declaretheorem[numberwithin=section,refname={Theorem,Theorems}]{theorem}
\declaretheorem[sibling=theorem,refname={Proposition,Propositions}]{proposition}

\declaretheorem[sibling=theorem,style=definition,refname={Definition,Definitions}]{definition}

\declaretheorem[sibling=theorem,style=remark,refname={Remark,Remarks}]{remark}

\newcounter{claimCounter}[theorem]

\theoremstyle{remark}

\newcounter{LRClaim}

\declaretheorem[numberlike=LRClaim,style=remark,name=Claim,refname={Claim,Claims}]{lrclaim}
\declaretheorem[numberlike=LRClaim,style=remark,name=Definition,refname={Definition,Definitions}]{lrdef}

\newlist{equivalent}{enumerate}{1}
\setlist[equivalent,1]{label=\textup{(\arabic*)}}

\newlist{sublemma}{enumerate}{1}
\setlist[sublemma,1]{label=\textup{(\alph*)}}

\newlist{orderedlist}{enumerate}{1}
\setlist[orderedlist,1]{label=(\roman*)}

\newlist{sublemma*}{enumerate*}{1}
\setlist[sublemma*,1]{label=\textup{(\alph*)},afterlabel=\hspace{5pt}}

\newlist{orderedlist*}{enumerate*}{1}
\setlist[orderedlist*,1]{label=\textup{(\roman*)},afterlabel=\hspace{3pt}}

\newlist{TSPlist}{enumerate}{1}
\setlist[TSPlist,1]{labelindent=0em,labelwidth=4em,labelsep=0.5em,leftmargin=5em,itemindent=!,
label=\textup{TSP(\arabic*):},ref=\textup{TSP(\arabic*)}}

\begin{document}

\title[Iterated priority arguments in descriptive set theory]{Iterated priority arguments in Descriptive Set Theory}

\author[A.\:Day]{Adam Day}

\author[N.\:Greenberg]{Noam Greenberg}
\address{School of Mathematics and Statistics\\ Victoria University of Wellington\\ Wellington, New Zealand}
\email{greenberg@msor.vuw.ac.nz}

\author[M.\:Harrison-Trainor]{Matthew Harrison-Trainor}
\address{Department of Mathematics\\ University of Michigan}
\email{matthhar@umich.edu}

\author[D.\:Turetsky]{Dan Turetsky}
\address{School of Mathematics and Statistics\\ Victoria University of Wellington\\ Wellington, New Zealand}
\email{dan.turetsky@vuw.ac.nz}

\begin{abstract}
We present the true stages machinery and illustrate its applications to descriptive set theory.  We use this machinery to provide new proofs of the Hausdorff-Kuratowski and Wadge theorems on the structure of $\*\Delta^0_\xi$, Louveau and Saint-Raymond's separation theorem, and Louveau's separation theorem.
\end{abstract}

\maketitle

%%%%%%%%%%%%%%%%%%%%%%%%%%%%%%%%%%%%%%%%%%%%%%%%%
%%%%%%%%%%%%%%%%%%%%%%%%%%%%%%%%%%%%%%%%%%%%%%%%%
\section{Introduction}
%%%%%%%%%%%%%%%%%%%%%%%%%%%%%%%%%%%%%%%%%%%%%%%%%
%%%%%%%%%%%%%%%%%%%%%%%%%%%%%%%%%%%%%%%%%%%%%%%%%

Ever since Mostowski \cite{Mostowski:1946} and Addison \cite{Addison:59:Separation} observed the connections between Kleene's work on computability and the hyperarithmetic hierarchy \cite{Kleene:1943,Kleene:55:Souslin} and Lusin and Souslin's study of Borel and analytic sets \cite{lusin:17,souslin:17}, effective methods have been widely utilised in descriptive set theory. The computable theory gives alternative proofs of classical arguments, for example Addison's proof of Kondo's co-analytic uniformisation \cite{kondo:38}, Sack's proof of the measurability of analytic sets (see \cite[Lem.II.6.2]{SacksBook}), or Solovay's effectivised Galvin-Prikry theorem~\cite{Solovay:78:HyperarithmeticallyEncoded}. More recently, effective tools were used for the $E_0$-dichotomy for Borel equivalence relations (Harrington, Kechris and Louveau~\cite{HarringtonKechrisLouveau}) and the $G_0$-dichotomy for Borel graph colourings (Kechris, Solecki and Todorcevic~\cite{KechrisSoleckiTodorcevic}). Perhaps the most prominent interaction of the effective and classical hierarchies appear in Louveau's separation theorem \cite{Louveau:80:Separation}, and other work of Loueveau, Saint-Raymond and Debs \cite{LouveauSR:WH,LouveauSR:Strength,Debs:87}. 

Beyond the basic understanding of the hyperarithmetic hierarchy as a refinement of the Borel one, a number of tools involving computability theory have been used in these studies, for instance computability on admissible sets, and the Gandy-Harrington topology. However, the most unique and central technique of computability, the priority method initiated by Friedberg \cite{Friedberg:Priority} and Muchnik \cite{Muchnik:56}, appeared to have limited application to questions of set theory. Martin's original proof of Borel determinacy \cite{Martin:BorelDeterminacy} had a finite-injury aspect to it; recent work by Lecomte and Zeleny \cite{LecomteZeleny} involves an infinite-injury argument. These appeared to be isolated incidents. 

This changed very recently, first with work of Day, Downey and Westrick studying topological degrees of discontinuous functions \cite{DayDowneyWestrick}, and then with Day and Marks's resolution of the decomposability conjecture \cite{DayMarks}. Both of these used iterated priority arguments, in particular, as formulated by Montalb\'an's \emph{true stages} machinery. 

Iterated priority arguments originated in works of Harrington (unpublished, see \cite{Knight:Workers:Finite,Knight:Worker:Transfinite}) and Ash (\cite{Ash:86:meta}, see \cite{AshKnight:book}) in computable structure theory. These are complex priority arguments whose iterated nature reflects the hyperarithmetic hierarchy. They are used to construct computable objects while satisfying some ordinal-height system of requirements and guessing at non-computable information. The idea is to break such a construction down into smaller, more tractable pieces. A typical application is the Ash-Watnick theorem \cite{Ash:90:Watnick}, which states that if~$\delta$ is a computable ordinal and~$\+L$ is a~$\Delta^0_{2\delta+1}$ linear ordering, then $\Int^\delta \cdot \+L$ has a computable copy. Here we guess at what~$\+L$ is, while building approximations to an iteration (of length~$\delta$) of the Hausdorff derivative of our copy of $\Int^\delta\cdot \+L$. Another style of iterated priority arguments has been formulated by Lempp and Lerman \cite{LemppLerman:abstract,Lerman:Framework:Book}, who use sequences of trees of strategies to unfold complicated requirements. 

Montalb\'an \cite{Montalban:TrueStages:paper} extended Ash's metatheorem and presented it in a dynamic way, using the concept of true stages. This concept generalises Dekker's non-deficiency stages, those which provide a correct approximation of the halting problem. Perhaps surprisingly, using a relativised version, true stages turn out to also be of use in constructing functions on Baire space and other Polish spaces, as was demonstrated by Day, Downey and Westrick, and Day and Marks. In some sense, however, the history of such applications goes back several decades: Louveau and Saint-Raymond developed a techinque called the ramification method to prove Borel Wadge determinacy~\cite{LouveauSR:WH,LouveauSR:Strength}. An examination of this technique reveals that it is fundamentally a ``worker argument'', Harrington's formulation of his iterated priority arguments. 

The purpose of this paper is to describe the general {true stages} machinery and illustrate its applications to descriptive set theory by providing new proofs of several results. We start with results which do not actually require a priority argument, but for which the machinery gives smooth proofs regardless: we discuss changes of topology, the Hausdorff-Kuratowski theorem on the structure of $\*\Delta^0_{\xi+1}$, and Wadge's theorem on the structure of $\*\Delta^0_{\lambda}$, for $\lambda$ limit. We then show how to use true stages for a priority argument, and give a proof of Louveau's separation theorem. 

These studies have applications to reverse mathematics. The question of the axiomatic strength of results in descriptive set theory was raised by Louveau and Saint-Raymond. In \cite{LouveauSR:Strength} they showed that Borel Wadge determinacy is provable in second-order arithmetic.
% (for more details, see Lubarsky's review \cite{Lubarsky:Review}). 
This was quite surprising, since the most straightforward proof relies on Borel determinacy, which is known to require strong axioms \cite{FriedmanH:BorelDeterminacy}; and since $\bPi^1_1$ Wadge determinacy was known to be equivalent to full $\bPi^1_1$ determinacy (Harrington \cite{Harrington:AnalyticDeterminacy}); the same holds for $\bPi^1_2$ (Hjorth \cite{Hjorth:Wadge:Pi12}). The ramification method allowed Louveau and Saint-Raymond to avoid the reliance on Borel determinacy and rather, reduce the Borel Wadge games to closed games. Similarly, the standard proof of Louveau's separation theorem uses the Gandy-Harrington topology, and can be carried out using the $\Pi^1_1$-comprehension axiom system. Our proof shows that in fact, the weaker system $\ATR_0$ (arithmetical transfinite recursion) suffices. 

This paper is intended as one of a pair. In the companion paper \cite{companionpaper}, we give a new and effective classification of all Borel Wadge classes, using the true stages machinery. A corollary is a proof of Borel Wadge determinacy in $\ATR_0 + \Sigma^1_1$-$\IND$. One of the steps is a proof of the Louveau and Saint-Raymond separation theorem \cite{LouveauSR:WH} for all Borel Wadge classes. In the current paper we give the simpler argument, for the classes $\bSigma^0_\xi$, from which we derive the Louveau separation theorem.

\subsection{A note on subscripts}

In the hyperarithmetic hierarchy, statements about the finite levels of the hierarchy often have an ``off by one'' error when generalized to the infinite levels.  This is illustrated by (and can be viewed as originating from) the following pair of results:

\begin{proposition}[Post~\cite{Post:48:report}]
For $n < \omega$ and a set $X \subseteq \omega$, $X \le_T \emptyset^{(n)}$ if and only if $X \in \Delta^0_{n+1}$.
\end{proposition}

\begin{proposition}[Ash, see~\cite{AshKnight:book}]
For $\omega \le \alpha < \wock$ and a set $X \subseteq \omega$, $X \le_T \emptyset^{(\alpha)}$ if and only if $X \in \Delta^0_\alpha$.  
\end{proposition}

Observe that the subscript in the first result contains ``$+1$'', while the subscript in the second does not.  Ultimately, this comes down to the fact that computable corresponds to $\Delta^0_1$ rather than $\Delta^0_0$. Similarly, the class of open sets is denoted~$\bSigma^0_1$ rather than~$\bSigma^0_0$, which contrasts, for example with the Baire hierarchy of Borel functions.  This can be unified, however, by making use of $1+\alpha$.  Note that for $\alpha < \omega$, $1+\alpha = \alpha+1$, while for $\alpha \ge \omega$, $1+\alpha = \alpha$.  Thus the following holds:

\begin{proposition}
For $\alpha < \wock$ and a set $X \subseteq \omega$, $X \le_T \emptyset^{(\alpha)}$ if and only if $X \in \Delta^0_{1+\alpha}$.  
\end{proposition}

We will make extensive use of this device, generally in the subscripts of~$\Delta$'s and~$\Sigma$'s.

%%%%%%%%%%%%%%%%%%%%%%%%%%%%%%%%%%%%%%%%%%%%%%%%%
%%%%%%%%%%%%%%%%%%%%%%%%%%%%%%%%%%%%%%%%%%%%%%%%%
\section{Change of topology and true Stages}\label{sec:topology_and_true_stages}
%%%%%%%%%%%%%%%%%%%%%%%%%%%%%%%%%%%%%%%%%%%%%%%%%
%%%%%%%%%%%%%%%%%%%%%%%%%%%%%%%%%%%%%%%%%%%%%%%%%

A commonly used technique in descriptive set theory is the enrichment of the topology of a Polish space. For simplicity, we shall restrict ourselves in this paper to Baire space $\Baire = \w^\w$. The following is a version of \cite[Thm.13.1]{Kechris:book}:

\begin{proposition} \label{prop:change_of_topology}
    If $A\subseteq \Baire$ is Borel, then there is a Polish topology on $\Baire$, extending the standard one, and which has the same Borel sets, in which~$A$ is open.  
\end{proposition}

The proof given in \cite{Kechris:book} is direct, by induction on the Borel rank of~$A$. Alternatively, it can be deduced from a characterisation of classes of Borel sets in terms of generalised homeomorphisms, due to Kuratowski \cite{Kuratowski:33} and Sierpi{\'n}ski \cite{Sierpinski:33}. Recall that a function $f\colon Y\to X$ between Polish spaces is $\bSigma^0_\xi$-measurable if $f^{-1}[U]$ is $\bSigma^0_\xi$ for every open set $U\subseteq X$. A function $f\colon \Baire\to X$ is $\bSigma^0_{\xi+1}$-measurable if and only if it is Baire class~$\xi$ (see \cite[Thm.24.3]{Kechris:book}). Kuratowski and Sierpi{\'n}ski essentially showed: 

\begin{proposition} \label{prop:Kuratowksi_Sierpinski}
    Let $\xi \ge 1$. A set $A\subseteq \Baire$ is $\bSigma^0_{\xi}$ if and only if there is a closed set $E\subseteq \Baire$ and a bijection $f\colon \Baire\to E$ such that:
    \begin{orderedlist}
        \item $f$ is $\bSigma^0_\xi$-measurable;
        \item $f^{-1}$ is continuous; and
        \item $A = f^{-1}[B]$ for some open $B\subseteq \Baire$. 
    \end{orderedlist}
\end{proposition}

For a detailed account see \cite[Thm.IV.C.10]{Wadge:phd}. \Cref{prop:change_of_topology} can be directly deduced from \cref{prop:Kuratowksi_Sierpinski} by taking the topology on~$\Baire$ which makes~$f$ a homeomorphism between~$\Baire$ and~$E$; we know that a closed subset of~$\Baire$ is Polish. 

\smallskip

An effective proof of \cref{prop:Kuratowksi_Sierpinski} was observed by A.~Marks. For any $x\in \Baire$ and ordinal $\alpha<\w_1^x$ (i.e., $x$ computes a copy of~$\alpha$), we let $x^{(\alpha)}$ denote the the iteration of the Turing jump operator (the relativised halting problem) along~$\alpha$, starting with~$x$. Spector showed \cite{Spector:Uniqueness} that the Turing degree of $x^{(\alpha)}$ does not depend on the presentation of~$\alpha$ as an $x$-computable well-ordering. The degree of $x^{(\alpha)}$ is Turing complete for the $\Delta^0_{1+\alpha}(x)$ sets. Under standard definitions, $\{ x^{(\alpha)}\,:\, x\in \Baire \}$ is a $\Pi^0_2$ subset of Cantor space. It can, however, be pulled back to a $\Pi^0_1$ subset of Baire space: for each~$x$, $x^{(\alpha)}$ is a $\Pi^0_2(x)$ singleton, and each $\Pi^0_2(x)$ singleton in Cantor space is Turing equivalent to a $\Pi^0_1(x)$ singleton in Baire space. This is all uniform in~$x$. After this renaming, the function $x\mapsto x^{(\alpha)}$ is $\bSigma^0_{1+\alpha}$-measurable with closed image and continuous inverse, and the following holds:

\begin{proposition} \label{prop:effective_Kuratowski_Sierpinski:1}
    Let $\alpha <\wock$. A set $A\subseteq \Baire$ is $\Sigma^0_{1+\alpha}$ if and only if there is a $\Sigma^0_1$ set $V\subseteq \Baire$ such that $A = \left\{ x \,:\,  x^{(\alpha)}\in V \right\}$. 
\end{proposition}

The relativised jump functions $x\mapsto (x,z)^{(\alpha)}$ are universal: a function $f\colon \Baire\to X$ is $\bSigma^0_{1+\alpha}$-measurable if and only if there is some oracle (parameter)~$z$ and some continuous funciton $g\colon \Baire\to X$ such that $f(x) = g((x,z)^{(\alpha)})$. Thus, \cref{prop:Kuratowksi_Sierpinski} can be deduced from its effective version \cref{prop:effective_Kuratowski_Sierpinski:1}. We can also deduce an effective version of \cref{prop:change_of_topology}. For every oracle~$z$ and every $\alpha < \w_1^z$, let the \emph{$(z,\alpha)$-topology} on~$\Baire$ be the topology generated by the $\Sigma^0_{1+\alpha}(z)$ sets. It extends the standard topology on~$\Baire$, and has the same Borel sets. \Cref{prop:effective_Kuratowski_Sierpinski:1} implies the following, which in turn implies \cref{prop:change_of_topology}:

\begin{proposition} \label{prop:the_z_xi_topology_is_Polish}
    For every $z$ and every $\alpha< \w_1^z$, the $(z,\alpha)$-topology is Polish. 
\end{proposition}

%%%%%%%%%%%%%%%%%%%%%%%%%%%%%%%%%%%%%%%%%%%%%%%%%
\subsection{Enter true stages} 
\label{sub:enter_true_stages}
%%%%%%%%%%%%%%%%%%%%%%%%%%%%%%%%%%%%%%%%%%%%%%%%%

Baire space is the set of infinite paths through the tree $(\w^{<\w},\preceq)$. A metric witnessing that~$\Baire$ is Polish is derived from the tree: $d(x,y) = 2^{-|\s|}$, where $\s$ is the greatest common initial segment of~$x$ and~$y$ on the tree. Our first application of the true stages machinery will be an extension of this idea to the $(z,\alpha)$-topology in an internally coherent fashion. For simplicity of notation, we state the unrelativised version ($z = \emptyset$). 

\medskip

The true stages machinery provides, for each computable ordinal $\alpha < \wock$, a partial ordering $\preceq_\alpha$ on $\w^{\le \w}$ with a variety of useful properties. %\footnote{We are being somewhat imprecise here; the relation $\preceq_\alpha$ depends, in fact, on a choice of computable well-ordering isomorphic to~$\alpha$.} 
We will list the properties that we will need as they become relevant. We start with the following four. 
    \begin{TSPlist}
        \item \label{TSP:extends_prec} For $\s,\tau\in \w^{\le \w}$, $\s\preceq_\alpha \tau$ implies $\s\preceq \tau$. For $\alpha=0$, $\s\preceq_0 \tau\Iff \s\preceq\tau$.

        \item  \label{TSP:tree}
        $(\w^{\le \w}, \preceq_\alpha)$ is a tree: for all $\tau\in \w^{\le \w}$, $\left\{ \s \,:\,  \s\preceq_\alpha \tau \right\}$ is linearly ordered; the root of the tree is $\emptystring$ (the empty sequence). 

       % \item \label{TSP:computable}
        % The restriction of $\preceq_\alpha$ to $\w^{<\w}$ is computable.

        \item \label{TSP:unique_path}
        For every $x\in \Baire$, $\left\{ \s\in \w^{<\w} \,:\, \s\prec_\alpha x   \right\}$ is the unique infinite path of the restriction of $\preceq_\alpha$ to $\left\{ \s\in \w^{<\w} \,:\,  \s\prec x \right\}$.

        \item \label{TSP:Sigma_alpha_sets}
        A set $A\subseteq \Baire$ is $\Sigma^0_{1+\alpha}$ if and only if there is a c.e.\ set $U\subseteq \w^{<\w}$ such that 
        \[
            A = [U]^{\prec}_{\alpha} = \left\{ x\in \Baire \,:\,  (\exists \s\in U)\,\,\s\prec_\alpha x \right\}.
        \]
    \end{TSPlist}

\Cref{prop:the_z_xi_topology_is_Polish} for $z=\emptyset$ follows: we let $d_\alpha(x,y) = 2^{-|\s|}$, where~$\s$ is the longest string satisfying $\s\prec_\alpha x, y$. The open sets of the $(\emptyset,\alpha)$-topology are precisely the sets $[U]^\prec_\alpha$ for any $U\subseteq \w^{<\w}$. \ref{TSP:extends_prec} implies that this generalises the usual topology. In general, for each oracle~$z$, we can relativise the machinery to~$z$ and obtain, for each $\alpha < \w_1^z$, a partial ordering $\preceq^z_\alpha$ with the same properties, except that in \ref{TSP:Sigma_alpha_sets} we replace c.e.\ by $z$-c.e.\ and $\Sigma^0_{1+\alpha}$ by $\Sigma^0_{1+\alpha}(z)$.

\smallskip

As this is an expository paper, we will relegate to the companion paper \cite{companionpaper} the details of the construction of these partial orderings and the verification of their properties. Montalb\'an first developed his true stages machinery in \cite{Montalban:TrueStages:paper}, based on his work with Marcone in \cite{MarconeMontalban:Veblen} on the Veblen functions and the iterated Turing jump. Greenberg and Turetsky then gave an alternative development in \cite{GreenbergTuretsky:Pi11}. Both of these were restricted to $x = 0^\infty$. Day, Downey and Westrick, and Greenberg and Turetsky independently observed that both of these versions can be extended to all of Baire space. 

Let us informally describe the main ideas. The relations $\preceq_\alpha$ are defined by recursion on~$\alpha$: we first need to define $\preceq_\beta$ for all $\beta<\alpha$. To make the construction work, these relations need to be internally coherent, for example, they are nested and continuous:
\begin{TSPlist}[resume]
    \item \label{TSP:nested}
    If $\alpha\le \beta$, then $\s\preceq_\beta \tau$ implies $\s\preceq_\alpha \tau$; 

    \item \label{TSP:continuous}
    If $\lambda$ is a limit ordinal, then $\s\preceq_\lambda \tau$ if and only if for all $\alpha< \lambda$, $\s\preceq_\alpha \tau$. 
\end{TSPlist}

The idea is to define, for each \emph{finite} sequence~$\s$, a finite sequence $\s^{(\alpha)}$, which is $\s$'s guess of an initial segment of the $\alpha$-jump of infinite sequences~$x$ extending~$\s$. This guess is independent of the choice of $x\succ \s$. The guess may be correct for some choices of~$x$ and incorrect for others. Roughly, we let $\s\prec_\alpha x$ if $\s^{(\alpha)}\prec x^{(\alpha)}$, i.e., when~$\s$'s guess is correct for~$x$; we say that such a $\s$ is an \emph{$\alpha$-true} initial segment of~$x$. One of the main ideas, though, is the extension of this relation to a relation between two \emph{finite} sequences $\s\preceq \tau$. Again, the idea is to let $\s\preceq_\alpha \tau$ if $\s^{(\alpha)}\preceq \tau^{(\alpha)}$: the guesses of~$\s$ and~$\tau$ about the $\alpha$-jump do not contradict each other. We say that~$\s$ \emph{appears to be $\alpha$-true} to~$\tau$. The terminology ``true \emph{stage}'' comes from the application to computable constructions, as we discuss below. The idea is that in an $x$-computable construction, $x\rest{s}$ is the stage~$s$ information we have, and $(x\rest{s})^{(\alpha)}$ is our stage~$s$ guess about $x^{(\alpha)}$; $s$ is an $(x,\alpha)$-true stage if $x\rest{s}\prec_\alpha x$. 

This definition of $\preceq_\alpha$ is not quite right. It needs to be modified in order to deal with limit ordinals. In this sketch we will ignore this modification; the definition  $\s\preceq_\alpha \tau \Iff \s^{(\alpha)}\preceq \tau^{(\alpha)}$ conveys the main idea. The nested condition \ref{TSP:nested} results from the fact that since the Turing jump operator is defined by recursion on the ordinal, in order to compute $\s^{(\beta)}$, we need to first, in some sense, compute $\s^{(\alpha)}$ for all $\alpha<\beta$. \ref{TSP:continuous} reflects the fact that $x^{(\lambda)}$ is Turing equivalent to the infinite join $\bigoplus_{\alpha<\lambda}x^{(\alpha)}$. 

The main difference between the developments in \cite{Montalban:TrueStages:paper} and \cite{GreenbergTuretsky:Pi11} is that the former first defines the strings~$\s^{(\alpha)}$ and then modifies the resulting relations~$\preceq_\alpha$, whereas the latter gives a simultaneous inductive definition of both~$\s^{(\alpha)}$ and~$\preceq_\alpha$. 

\subsubsection{A single step}
Here we sketch out a single step, i.e., the construction of $\s^{(\alpha+1)}$ from $\s^{(\alpha)}$.  The reader not interested in peeking under the hood may safely skip down to~\ref{TSP:p_function}.

For each~$x$ and~$\alpha$, $x^{(\alpha+1)}$ is an element of Baire space which is Turing equivalent to $(x^{(\alpha)})'$, the halting problem relative to $x^{(\alpha)}$. Fix a universal oracle Turing machine~$M$ which enumerates the jump: for all~$x$ and~$e$, $e\in x'$ if and only if $M^x(e)\converge$, i.e., the machine~$M$ with oracle~$x$ halts on input~$e$. For a finite~$\s$, we let $\s'$ be the collection of~$e$ such that $M^\s_{|\s|}(e)\converge$, i.e., those~$e$ for which the machine~$M$, with oracle~$\s$, halts on input~$e$ in at most~$|\s|$ many steps. 

For each finite sequence $\s\in \w^{<\w}$, the sequence $\s^{(\alpha+1)}$ will encode a finite initial segment of $(\s^{(\alpha)})'$: for some $p_\alpha(\s)\in \Nat$, for all $e<p_\alpha(\s)$, $\s^{(\alpha+1)}$ tells us whether $e\in (\s^{(\alpha)})'$ or not. The idea is the following: if $e\in (\s^{(\alpha)})'$ then~$\s$ can be confident that~$e$ is in the jump -- for all infinite~$x$, if $\s\prec_\alpha x$ then $e\in (x^{(\alpha)})'$. For $e<p_\alpha(\s)$ such that $e\notin (\s^{(\alpha)})'$, $\s$ is sufficiently brave to declare that $e\notin (x^{(\alpha)})'$ for $x\succ_\alpha \s$; it may be correct about some such~$x$'s, but not about others. For $e\ge p_\alpha(\s)$, $\s$ makes no commitment about $(x^{(\alpha)})'(e)$. 

If~$\s\prec_\alpha \tau$, then as $\s^{(\alpha)}\preceq \tau^{(\alpha)}$, we have $(\s^{(\alpha)})'\subseteq (\tau^{(\alpha)})'$: if the machine~$M$ with oracle $\s^{(\alpha)}$ halts on~$e$ in at most $|\s^{(\alpha)}|$ many steps, then this also holds when the oracle is extended to $\tau^{(\alpha)}$ (and more steps are allowed). We let $\s\preceq_{\alpha+1} \tau$ if $\s\prec_\alpha \tau$ and~$\tau$ has no proof that~$\s$ was wrong about $(\s^{(\alpha)})'\rest{p_\alpha(\s)}$: if for all $e<p_\alpha(\s)$, if $e\notin (\s^{(\alpha)})'$ then also $e\notin (\tau^{(\alpha)})'$. 

Fixing $x\in \Baire$, our goal is to ensure, for each~$\alpha$, that:
\begin{description}
    \item[($*$)] $x^{(\alpha)} = \bigcup \left\{ \s^{(\alpha)} \,:\, \s\prec_\alpha x   \right\}$. 
\end{description}
That is, there are infinitely many $\alpha$-true initial segments of~$x$, and as they get longer, they give us more and more information about $x^{(\alpha)}$. Suppose that~($*$) is known for some~$\alpha$. To ensure that~($*$) also holds for~$\alpha+1$, we need to define $p_\alpha(\s)$ wisely. The main idea is that of the ``non-definciency stages'' of Dekker, in his construction of a hypersimple set in every c.e.\ degree \cite{Dekker:Hypersimple}. We fix a enumerations of the jump sets~$\s'$ so that if $\s\preceq \tau$ then the enumeration of~$\tau'$ extends that of~$\s'$; and we let $p_\alpha(\s)$ be the last number enumerated into~$(\s^{(\alpha)})'$. So~$\s$ observes that the number $p = p_\alpha(\s)$ has just entered $(\s^{(\alpha)})'$; it thinks that no smaller numbers will enter the jump later, but is not willing to give an opinion about larger ones. To verify~($*$), for any $\s\prec_\alpha x$, let $k$ be the least element of $(x^{(\alpha)})'\setminus (\s^{(\alpha)})'$. By~($*$) for~$\alpha$, for sufficiently long $\tau\prec_\alpha x$ we have $k\in (\tau^{(\alpha)})'$; the least such~$\tau$ will have $k = p_\alpha(\tau)$ and so $\tau \prec_{\alpha+1}x$. 

While the actual details are a little different, we can now state another important property of the true stages machinery:
\begin{TSPlist}[resume]
     \item \label{TSP:p_function}
     There are functions $p_\alpha\colon \w^{<\w}\to \Nat$ such that for all $\s,\tau\in \w^{<\w}$, $\s\prec_{\alpha+1}\tau$ if and only if $\s\prec_\alpha\tau$ and for all finite~$\rho$ satisfying $\s\prec_\alpha \rho\preceq_{\alpha}\tau$ we have $p_\alpha(\rho)\ge p_\alpha(\s)$. 
 \end{TSPlist}  

($*$) can also explain \ref{TSP:Sigma_alpha_sets}; we give a sketch.

\begin{proof}[Sketch of derivation of \ref{TSP:Sigma_alpha_sets}:]
    In one direction, let $A$ be $\Sigma^0_{1+\alpha}$; let~$V$ be given by \cref{prop:effective_Kuratowski_Sierpinski:1}. There is a c.e.\ set of strings~$V_0$, closed under taking extensions, that generates~$V$ as an open set. By~($*$),  $U = \left\{ \s\in \w^{<\w} \,:\, \s^{(\alpha)}\in  V_0  \right\}$ is as required for \ref{TSP:Sigma_alpha_sets}. In the other direction, for each~$x$, $\left\{ \s \,:\,  \s\prec_\alpha x \right\}$ is $x^{(\alpha)}$-computable, uniformly in~$x$. That is, there is some~$\Delta^0_1$ set $W\subset \w^{<\w}\times \Baire$ such that $\s\prec_\alpha x \Iff (\s,x^{(\alpha)})\in W$. \Cref{prop:effective_Kuratowski_Sierpinski:1} then implies that $[\s]^\prec_\alpha$ is $\Delta^0_{1+\alpha}$, uniformly in~$\s$. We mention that we have implicitly used \ref{TSP:computable}, stated below. We also remark that \ref{TSP:Sigma_alpha_sets} is uniform: we can effectively pass from $\Sigma^0_{1+\alpha}$ indices of~$A$ to c.e.\ indices of~$U$.
\end{proof}

What we have not discussed so far is how to ensure that~($*$) holds for limit ordinals~$\alpha$. This is, in fact, the most difficult aspect of the construction of the true stages machinery. In \cite{GreenbergTuretsky:Pi11}, this is solved by the particular encoding of the iterated jumps into $\s^{(\alpha)}$; a kind of diagonal intersection argument is used. We will further discuss limit levels in the next section.  

For a final remark, we observe that like the \emph{set} $x^{(\alpha)}$ (and unlike its Turing degree), the partial orderings $\preceq_\alpha$ actually depend on the choice of a computablewell-ordering of~$\Nat$ of order-type~$\alpha$. For this reason, we cannot define~$\preceq_\alpha$ in a way which will satisfy \ref{TSP:nested},\ref{TSP:continuous} and other nice properties for all computable ordinals at once. In \cite{GreenbergTuretsky:Pi11}, this obstacle is overcome by an overspill argument, in which the true stages machinery is applied to a pseudo ordinal $\delta^*>\wock$. The price to pay then is having to ensure that these pseudo ordinals do not interfere with the intended construction.

%%%%%%%%%%%%%%%%%%%%%%%%%%%%%%%%%%%%%%%%%%%%%%%%%
\subsubsection{Iterated priority arguments} % (fold)
\label{ssub:iterated_priority_argument}
%%%%%%%%%%%%%%%%%%%%%%%%%%%%%%%%%%%%%%%%%%%%%%%%%

In order to apply true stages to iterated priority arguments, which are computable constructions, we require:
\begin{TSPlist}[resume]
    \item \label{TSP:computable}
     The restriction of the relations $\preceq_\alpha$ to $\w^{<\w}$ (i.e., to \emph{finite} sequences) is computable. So is the map $\s\mapsto \s^{(\alpha)}$ for finite~$\s$, and the function~$p_\alpha$ of \ref{TSP:p_function}. 
\end{TSPlist}

Note that in contrast, for each infinite~$x$, the relation $\s\prec_\alpha x$ is $\Delta^0_{1+\alpha}(x)$,  and cannot be any simpler. 

\smallskip

\ref{TSP:computable} allows us to use the guesses $\s^{(\alpha)}$ during a computable construction, and also observe, at every stage $t<\omega$, the opinion at stage~$t$ about the $\alpha$-truth of previous stages. 

Let us sketch how this is used in computable structure theory. As was our description of the development of the true stages machinery, this will be a very rough sketch, and it will not be essential for the remainder of the paper. Here we construct a single computable structure, so we apply the true stages machinery (as it was originally devised) to $x = 0^\infty$. For $s,t\le \w$, we write $s\le_\alpha t$ to denote $0^s\preceq_\alpha 0^t$, and say that~$s$ appears to be $\alpha$-true at~$t$; when $t = \w$, we say that $s$ is $\alpha$-true. 

Let~$\delta$ be a computable ordinal. Suppose that we wish to construct a computable structure $\+M$ and at the same time encode some $\emptyset^{(\delta)}$-computable information into the the computable $\Pi^0_\delta$-diagram of~$\+M$ (the relations on~$\+M$ defined by the computable $\forall_\delta$-fragment of $\+L_{\w_1,\w}$ in the language of~$\+M$). For example, in the Ash-Watnick theorem mentioned above, we are given a $\emptyset^{(2\delta)}$-computable linear ordering~$\+L$, and we need to construct a linear ordering~$\+M$ of order-type $\Int^\alpha\cdot \+L$; the $\Pi^0_{2\delta}$ relation that interests us is whether two elements of~$\+M$ are in the same copy of $\Int^\delta$, i.e., whether they are identified after applying the Hausdorff derivative (identify points which are finitely far apart)~$\delta$ many times.\footnote{The ordinal $2\delta$ is not quite correct; for $\delta\ge \w$ we need $\emptyset^{(2\delta+1)}$ to compute the iteration of the Hausdorff derivative of length~$\delta$. A modification of the true stages machinery is needed to overcome this problem, so that for levels $\alpha\ge \w$, the $\alpha$-true stages actually compute $\emptyset^{(\alpha+1)}$ rather than just $\emptyset^{(\alpha)}$. This is done while maintaining \ref{TSP:continuous}. The modification is undesirable for applications to descriptive set theory, as it makes \ref{TSP:Sigma_alpha_sets} fail at limit levels. Technically, the modification does not have \ref{TSP:limit} stated below.}  

During the construction, at each stage $s<\w$, we construct not only a finite substructure $\+M_s$ of the intended computable structure~$\+M = \+M_\w$, but also an approximation $\+D^\delta_s$ of the $\Pi^0_\delta$ diagram of~$\+M$. We use our guess $\emptyset_s^{(\delta)} = (0^s)^{(\delta)}$ of~$\emptyset^{(\delta)}$ to determine which statements to include in $\+D_s^\delta$. 

We require that $\+M_s\subseteq \+M_t$ when $s\le t$. Since the construction is computable, the sequence $\smallseq{\+M_s}$ is computable, and so $\+M = \+M_\w = \bigcup_s \+M_s$ is a computable structure, as required. Further, and this is the key point, we ensure that for $s \le_\delta t$, $\+D^\delta_s \subseteq \+D^\delta_t$. What is useful to us at the end is that when $s<_\delta \w$ ($s$ is a  $\delta$-true stage), $\+D^\delta_s \subset \+D^\delta_\w$, the latter being the true $\Pi^0_\delta$-diagram of~$\+M = \+M_\w$. For such~$s$, the choices we make for $\+D^\delta_s$ are correct, since our guess $\emptyset^{(\delta)}_s$ about $\emptyset^{(\delta)}$ is correct. Further, since $\emptyset^{(\delta)} = \bigcup \{\emptyset^{(\delta)}_s\,:\, s<_\delta \w\}$ (every $\delta$-correct piece of information is eventually revealed at $\delta$-true stages), we get $\+D^\delta_\w = \bigcup \{ \+D^\delta_s   \,:\, s<_\delta \w\}$. That is, every $\Pi^0_\delta$ fact is eventually correctly decided on the $\delta$-true stages. Note that the entire construction is computable: the map $s\mapsto \emptyset^{(\delta)}_s$, and so the map $s\mapsto \+D^\delta_s$, are computable. The reason that $\+D^\delta_\w$ is not computable is that $\{s\,:\, s<_\delta \w\}$ is not computable; the complexity is encoded in the set of $\delta$-true stages. Since that set is not computable, we need to ensure that $\+D^\delta_s\subseteq \+D^\delta_t$ when $s\le_\delta t$, even if~$s$ is not~$\delta$-true. 

One might wonder how we ensure that $\bigcup \left\{ \+D^\delta_s \,:\,  s<_\delta \w \right\}$ is in fact the $\Pi^0_\delta$-diagram of~$\+M$. For example, at $\delta$-true stages of the Ash-Watnick construction, we may declare that two elements~$a$ and~$b$ of~$\+M$ are in the same copy of $\Int^\delta$ (or not). How do we ensure that this declaration is in fact correct in the structure~$\+M$? Note that at each stage~$s$, $\+M_s$ is a finite linear ordering, so the declaration at that stage is just that: a declaration of intention, a promise about how the structure will be built from now on. For this purpose, we in fact not only make promises about the $\Pi^0_\delta$-diagram, but for all $\alpha\le \delta$, we give an approximation $\+D^\alpha_s$ of the $\Pi^0_\alpha$-diagram of~$\+M$. In the Ash-Watnick example, at level $2\alpha$ we declare which points are identified after~$\alpha$ many iterations of the Hausdorff derivative; at level $2\alpha+1$ we decide the successor relation on the $\alpha\tth$-derivative ordering. Each $\+D^\alpha_s$ is decided based on the approximation $\emptyset^{(\alpha)}_s$ of $\emptyset^{(\alpha)}$. The $\Pi^0_{\alpha+1}$-diagram of a structure can be recovered from the $\Pi^0_\alpha$-diagram; at every stage $s$, we build $\+D^\alpha_s$ so that $\+D^{\alpha+1}_s$ is consistent with the diagram recovered in this fashion from $\+D^\alpha_s$. We then let $\+D^\alpha_\w = \bigcup \left\{ \+D^\alpha_s \,:\,  s<_\alpha \w \right\}$, and inductively show that $\+D^\alpha_\w$ is indeed the correct diagram at its level. (We are eliding how limit levels are dealt with.) 

The overall resulting requirement for the construction is: for all $\alpha\le\delta$, if $s\le_\alpha t$ then $\+D^\alpha_s \subseteq \+D^\alpha_t$. (The requirement $\+M_s\subseteq \+M_t$ is incorporated into this, as $\+D^0_s$ is essentially the atomic diagram of~$\+M_s$; and $s\le_0 t$ iff $s\le t$.) As $\le_{\alpha+1}$ branches more than $\le_\alpha$, we will sometimes have the following scenario: $r <_\alpha s <_\alpha t$ with $r <_{\alpha+1} t$ and $s \not <_{\alpha+1} t$.  Then we must have $\+D^\alpha_r \subseteq \+D^\alpha_s \subseteq \+D^\alpha_t$; however $\+D^\alpha_s$ was built to support $\+D^{\alpha+1}_s$, and it may be that $\+D^{\alpha+1}_s \not \subseteq \+D^{\alpha+1}_t$.  So our construction must be such that our work for level $\alpha$ at stage $s$ can be folded into our work at stage $t$.  Arguing this is generally the main labor for a true stages priority argument.  One advantage we have is the following property, denoted $(\Club)$ by Montalb\'an:
\begin{TSPlist}
    \item[($\Club$)]  \label{TSP:Club}
    If $\sigma_0 \preceq_{\alpha} \sigma_1 \preceq_\alpha \sigma_2$ and $\sigma_0 \preceq_{\alpha+1} \sigma_2$, then $\sigma_0 \preceq_{\alpha+1} \sigma_1$.
\end{TSPlist}
In our scenario, it follows that $r <_{\alpha+1} s$, and so $\+D_r^{\alpha+1} \subseteq \+D_s^{\alpha+1}$, so the work for~$\alpha$ at stage~$s$ at least was not violating $\+D_r^{\alpha+1}$. The property  \RefClub{} follows immediately from \ref{TSP:p_function}. Informally, if $\s_0^{(\alpha)}\preceq \s_1^{(\alpha)} \preceq \s_2^{(\alpha)}$, and $\s_2$ has no evidence that~$\s_0$ was wrong about $(\s_0^{(\alpha)})'$, then~$\s_1$ cannot have any such evidence either.

%%%%%%%%%%%%%%%%%%%%%%%%%%%%%%%%%%%%%%%%%%%%%%%%%
%%%%%%%%%%%%%%%%%%%%%%%%%%%%%%%%%%%%%%%%%%%%%%%%%
\section{Analysis of the ambiguous Borel classes}
\label{sec:Delta_classes}
%%%%%%%%%%%%%%%%%%%%%%%%%%%%%%%%%%%%%%%%%%%%%%%%%
%%%%%%%%%%%%%%%%%%%%%%%%%%%%%%%%%%%%%%%%%%%%%%%%%

In this section we show how the true stages machinery allows us to give intuitive proofs of two theorems analysing the structure of the classes $\bDelta^0_\xi$, in terms of how they are built from lower-level classes. For successor~$\xi$, the Hausdorff-Kuratowski theorem gives an answer in terms of the Hausdorff difference hierarchy. For limit~$\xi$, Wadge gave an answer involving iterated partitioned unions.

%%%%%%%%%%%%%%%%%%%%%%%%%%%%%%%%%%%%%%%%%%%%%%%%%
\subsection{The Hausdorff-Kuratowski theorem} % (fold)
\label{sub:the_hausdorff_kuratowski_theorem}
%%%%%%%%%%%%%%%%%%%%%%%%%%%%%%%%%%%%%%%%%%%%%%%%%

The Hausdorff difference hierarchy (sometimes also named after Lavrentiev) is a transfinite extension of a hierarchy of finite Boolean operations. 

\begin{definition} \label{def:difference_hierarchy}
Let $\bGamma$ be a boldface pointclass and let $\eta\ge 1$ be a countable ordinal. We let $D_\eta(\bGamma)$ be the class of all sets of the form 
\[
A = \bigcup \Big\{   (A_i \setminus \bigcup_{j < i} A_j) \, : \, i < \eta \andd \text{parity}(i) \neq \text{parity}(\eta)\Big\} 
\]
where $\seq{A_{i}}_{i<\eta}$ is an increasing sequence of sets from~$\bGamma$. 
\end{definition}
Thus, $D_1(\bGamma) = \bGamma$, $D_{\eta+1}(\bGamma)$ is the collection of sets of the form $A\setminus B$ where $A\in \bGamma$ and $B\in D_\eta(\bGamma)$, and for limit~$\eta$, $D_\eta(\bGamma)$ is the collection of sets of the form $\bigcup_{i<\eta} (A_{2i+1}\setminus A_{2i})$ where $\seq{A_i}_{\alpha<\eta}$ is as in the definition. 

We similarly define $D_\eta(\Gamma)$ for lightface pointclasses; here we require that the sequence $\seq{A_i}$ be uniformly in~$\Gamma$. As with the relations~$\prec_\alpha$, the class will actually depend on the choice of a computable copy of~$\eta$. The computability-theoretic analogue of the Hausdorff difference hierarchy is the Ershov hierarchy \cite{Ershov2}, which has the same definition, where~$\Gamma$ is the class of c.e.\ subsets of~$\Nat$. Ershov used the notation $\Sigma^{-1}_{\eta}$ for $D_\eta(\Sigma^0_1(\Nat))$. 

Hausdorff \cite[VIII.4]{Haudorff:Grundzige:49} showed that $\bDelta^0_2 = \bigcup_{\eta<\w_1} D_\eta(\bSigma^0_1)$; Kuratowski \cite[37.III]{Kuratowksi:Topology} then used \cref{prop:Kuratowksi_Sierpinski} to prove:

\begin{theorem}[Hausdorff-Kuratowski] \label{thm:Hausdorff-Kuratowski}
    For all $1\le \xi<\w_1$,
     \[
     \bDelta^0_{\xi+1} = \bigcup_{\eta<\w_1}D_\eta(\bSigma^0_\xi).
     \]
\end{theorem}

See also \cite[Thm.22.27]{Kechris:book}. The following effective version of the Hausdorff-Kuratwoski theorem implies \cref{thm:Hausdorff-Kuratowski} by relativising to an oracle. 

\begin{theorem} \label{thm:effectiveHausdorff-Kuratowski}
  For all computable $\xi\ge 1$,   
  \[
     \Delta^0_{\xi+1} = \bigcup_{\eta<\wock}D_\eta(\Sigma^0_\xi).
     \]
\end{theorem}

\begin{remark}
    The effective version of Hausdorff's theorem (\cref{thm:effectiveHausdorff-Kuratowski} for $\xi = 1$) was proved by Ershov for subsets of~$\Nat$ and by Selivanov~\cite{Selivanov:2003} for subsets of~$\Baire$; see also~\cite{Pauly}. We remark again, however, that the lightface class $D_\eta(\Gamma)$ heavily depends on the particular choice of computable copy of~$\eta$. The theorem says that for every $\Delta^0_{\xi+1}$ set~$A$ there is some computable well-ordering~$R$ such that $A\in D_R(\Sigma^0_\xi)$. 

    For subsets of~$\Nat$, the situation is particularly dire; Ershov showed that every~$\Delta^0_2$ subset of~$\Nat$ is in $D_R(\Sigma^0_1)$ for some computable copy~$R$ of~$\w$. The complexity of a set $A\in \Delta^0_2$ is coded into this copy of~$\w$, rather than the sequence of sets $\seq{A_i}_{i<\w}$; the copy of~$\w$ may be a ``bad copy'', in which the successor relation is not computable. 

    For subsets of Baire space, topological considerations preclude such an anomaly, as the hierarchy of classes $D_\eta(\bSigma^0_\xi)$ is proper, and this is witnessed by lightface sets. Still, the classes $D_\eta(\Sigma^0_\xi)$ are not as robust as the boldface ones. Louveau and Saint-Raymond's work in \cite{LouveauSR:Strength} implies a weakening of \cref{thm:effectiveHausdorff-Kuratowski} which is robust:
    \[
        \Delta^0_{\xi+1}(\Delta^1_1) = \bigcup_{\eta<\wock} D_\eta(\Sigma^0_\xi)(\Delta^1_1), 
    \]
    where for a class~$\Gamma$ we let $\Gamma(\Delta^1_1) = \bigcup_{z\in \Delta^1_1} \Gamma(z)$. The class $D_\eta(\Sigma^0_\xi)(\Delta^1_1)$ depends only on the ordinal~$\eta$ and not on the choice of a~$\Delta^1_1$ copy of~$\eta$; this is because any two~$\Delta^1_1$ copies of~$\eta$ are isomorphic by a~$\Delta^1_1$ isomorphism. 
\end{remark}

The standard definitions of the Borel classes $\bSigma^0_\xi$ and $\bDelta^0_\xi$, as well as \cref{def:difference_hierarchy}, are, in the parlance of computability theory, \emph{static}: sets in these classes are characterised by Boolean operations. Property \ref{TSP:Sigma_alpha_sets} of the true stages machinery allows us to view membership in a $\Sigma^0_{1+\alpha}$ set as the result of a \emph{dynamic} process: to determine whether $x\in A$, we search over the finite sequences $\s\prec_\alpha x$, and we declare ``yes'' when we find such~$\s$ enumerated into~$U$. 

For the ambiguous classes, the prototypical dynamic decision process is given by Shoenfield's ``limit lemma'' \cite{Shoenfield:LimitLemma}, which states that a function $F\colon \Nat\to \Nat$ is~$\Delta^0_2$ (equivalently, computable from the halting problem $\emptyset'$) if and only if it has a \emph{computable approximation}: a computable function $f\colon \Nat^2\to \Nat$ such that for all~$n$, $F(n)= \lim_s f(n,s)$, the limit taken with respect to the discrete topology on~$\Nat$. That is, for all~$n$, for all but finitely many~$s$, $F(n) = f(n,s)$. Each function $n \mapsto f(n,s)$ is the ``stage~$s$ guess'' of the values of~$F$. 

This can be extended to subsets of Baire space: a \emph{computable approximation} of a function $F\colon \Baire\to \Nat$ is a computable function $f\colon \w^{<\w}\to \Nat$ such that for all $x\in \Baire$, for all but finitely many $\s\prec x$, we have $F(x) = f(\s)$. That is, the sequence $f(x\rest{0})$, $f(x\rest{1}),\dots$ approximates $F(x)$ in the sense of Shoenfield. Shoenfield's limit lemma relativises uniformly, and so it shows that a function $F\colon \Baire\to \Nat$ has a computable approximation if and only if it is $\Delta^0_2$-measurable. Using true stages, we can extend this further up the Borel hierarchy.

\begin{definition} \label{def:alpha_computable_approximation}
    Let $F\colon \Baire\to \Nat$ and let $\alpha<\wock$. An \emph{$\alpha$-approximation} of~$F$ is a function $f\colon \w^{<\w}\to \Nat$ such that for all $x\in \Baire$, 
    \[
          F(x) = \lim_{\s\prec_\alpha x} f(\s)
      \]  
      (in the sense that for all but finitely many $\s\prec_\alpha x$ we have $F(x) = f(\s)$). 
\end{definition}

The $\alpha$-analogue of Shoenfield's limit lemma is the following. Recall that for a lightface class~$\Gamma$, a function $F\colon \Baire\to \Nat$ is $\Gamma$-measurable if the sets $F^{-1}\{n\}$ are in~$\Gamma$, uniformly in~$n$. 

\begin{proposition} \label{prop:alpha-limit_lemma}
    A function $F\colon \Baire\to \Nat$ is $\Delta^0_{1+\alpha+1}$-measurable if and only if it has a computable $\alpha$-approximation. 
\end{proposition}

\begin{proof}
    In one direction, let~$f$ be a computable $\alpha$-approximation of~$F$. For $\s\in \w^{<\w}$, let $|\s|_\alpha$ denote the height of~$\s$ in the tree $(\w^{<\w}, \preceq_{\alpha+1})$. For $n,k<\w$, let 
    \[
        A_{n,k}  = \bigcup \{ [\sigma]_{\alpha} \,:\, f(\s)=n \andd |\s|_\alpha = k  \}. 
    \]
    These sets are uniformly $\Sigma^0_{1+\alpha}$. For each~$k$, $\{A_{n,k}\,:\, n<\w\}$ partitions~$\Baire$, so the sets $A_{n,k}$ are in fact uniformly $\Delta^0_{1+\alpha}$, so for all~$m$, $\bigcap_{k\ge m} A_{n,k}$ is $\Pi^0_{1+\alpha}$ (again, uniformly). Now for all~$n$, $F(x)=n$ if and only if $x\in A_{n,k}$ for all but finitely many~$k$, so $F$ is $\Sigma^0_{1+\alpha+1}$-measurable, which implies that it is  $\Delta^0_{1+\alpha+1}$-measurable. 

    \medskip
    
    In the other direction, suppose that $F\colon \Baire\to \Nat$ is $\Delta^0_{1+\alpha+1}$-measurable. By \ref{TSP:Sigma_alpha_sets}, which applies uniformly, there are uniformly c.e.\ sets $U_n\subseteq \w^{<\w}$ such that for all $x\in \Baire$, $F(x)=n$ if and only if $x\in [U_n]^\prec_{\alpha+1}$. 

    We we may assume that: 
    \begin{sublemma}
        \item the sets~$U_n$ are upwards closed in $(\w^{<\w}, \preceq_{\alpha+1})$; 
        \item the sets $U_n$ are uniformly computable (rather than uniformly c.e.), and their union $\bigcup_n U_n$ is computable as well; and 
        \item The sets $U_n$ are pairwise disjoint.  
    \end{sublemma}
    This is a standard trick from computability theory, relying on the fact that the upwards closure of~$U_n$ in $(\w^{<\w},\preceq_{\alpha+1})$ generates the same $\Sigma^0_{1+\alpha+1}$ set. For each $\tau\in \w^{<\w}$, we let $s=|\tau|_{\alpha+1}$ be the height of~$\tau$ in the tree $(\w^{<\w}, \preceq_{\alpha+1})$. We then declare that~$\tau$ belongs to the modified~$U_n$ if $n<s$, some $\s\preceq_{\alpha+1}\tau$ is enumerated into~$U_n$ by stage~$s$, and this does not hold for any $m<n$.\footnote{%
    We would imagine that by their property with respect to~$F$, the sets~$U_n$ are naturally pairwise disjoint: if $\s\in U_{n}\cap U_m$ then $F(x)$ is both~$n$ and~$m$ for any $x\in [\s]^\prec_{\alpha+1}$, which appears impossible. However, this argument relies on $[\s]^\prec_{\alpha+1}$ being nonempty, which is not necessarily the case. Indeed, for any $\beta>0$, the property $[\s]^\prec_\beta= \emptyset$ is $\Pi^1_1$-complete, so we cannot computably ignore all such finite sequences~$\s$. However, if, during the computation of~$U_n$, we notice that we have $\s\in U_n\cap U_m$ for some $m<n$, this provides proof that $[\s]^\prec_{\alpha+1}=\emptyset$, in which case we are free to declare that it does not belong to~$U_n$ without harming the desired properties of the sets~$U_n$.% 
    } 

    Having guaranteed properties (a)--(c) above, we define $f(\s)=n$ for all $\s\in U_n$, and $f(\s)=0$ if $\s\notin \bigcup_n U_n$. Property~(b) implies that~$f$ is computable. We show that~$f$ is an $\alpha$-approximation of~$F$. Let $x\in \Baire$; let $n = F(x)$. There is some $\s\prec_{\alpha+1} x$ in~$U_n$.  Let~$\tau$ be such that $\s\preceq_\alpha \tau \prec_\alpha x$. By \RefClub, $\s\preceq_{\alpha+1} \tau$. Since~$U_n$ is $\preceq_{\alpha+1}$-upwards closed, $\tau\in U_n$, so $f(\tau)=n$. 
\end{proof}

\begin{remark} 
    \Cref{prop:alpha-limit_lemma} is \emph{inherently effective}. In the language of effective topology, it says that a function $F\colon \Baire\to \Nat$ has a computable $\alpha$-approximation if and only if it is effectively continuous with respect to the $(\emptyset,\alpha+1)$-topology. One could imagine that this is just the effective version of the following: a function $F\colon \Baire\to \Nat$ has an $\alpha$-approximation if and only if it is continuous with respect to the $(\emptyset,\alpha+1)$-topology. The proof above gives the right-to-left implication. However the ``easier'' direction may fail: if $f$ is an $\alpha$-approximation of~$F$, then~$F$ is $\Delta^0_{1+\alpha+1}(f)$-measurable; this does not imply $(\emptyset,\alpha+1)$-continuity. 

    We note, though, that the proof of \cref{thm:effectiveHausdorff-Kuratowski} only uses the ``harder'' direction. Thus, we could use the machinery to give a direct proof of \cref{thm:Hausdorff-Kuratowski} without passing through \cref{thm:effectiveHausdorff-Kuratowski}. This is the path we take with Wadge's theorem below. 
\end{remark}

Toward a proof of \cref{thm:effectiveHausdorff-Kuratowski}, we give a dynamic description of the class $D_\eta(\Sigma^0_{1+\alpha})$. Let~$A$ be the $D_\eta(\Sigma^0_{1+\alpha})$ set defined from the increasing sequence $\seq{A_{i}}_{i<\eta}$.  After taking~$\alpha$ jumps, the sets~$A_i$ can be thought of as being ``c.e.'' (\ref{TSP:Sigma_alpha_sets}). The dynamic process we envision for deciding if $x\in A$ is the following. We start by guessing that $x\notin A$. Once we see that $x\in \bigcup_{i<\eta} A_{i}$, at each stage~$s$, we find the least $i<\eta$ for which $x\in A_i$, and we guess that $x\in A$ if and only if the parity$(i)\ne$ parity$(\eta)$. In fact, this is a particular kind of an $\alpha$-approximation. 

% \begin{definition} \label{def:eta_bounded_xi_computable_approximation}
%     Let $\alpha$ be a computable ordinal and $\eta\ge 1$ be countable ordinals; let $f$ be an $\alpha$-approximation of a function $F\colon \Baire\to \{0,1\}$. 

%     An \emph{$\eta$-witness for the $\alpha$-convergence of~$f$} is a computable function $o\colon \w^{<\w}\to \eta+1$ such that for all $\s,\tau\in \w^{<\w}$, 
%     \begin{orderedlist}
%         \item If $\s\preceq_\alpha \tau$ then $o(\tau)\le o(\s)$; 
%         \item If $\s\preceq_\alpha \tau$ and $f(\tau)\ne f(\s)$ then $o(\tau)< o(\s)$;
%         \item If $o(\s)=\eta$ then $f(\s)=0$.  
%     \end{orderedlist}
% \end{definition}

\begin{proposition} \label{prop:D_eta_level_and_approximations}
    Let~$\alpha$ and $\eta\ge 1$ be computable ordinals. A set $A\subseteq \Baire$ is in $D_\eta(\Sigma^0_{1+\alpha})$ if and only if its characteristic function~$1_A$ has a computable $\alpha$-approximation~$f$ for which is there is a computable ``witness'' function $o\colon \w^{<\w}\to \eta+1$ satisfying:
    \begin{orderedlist}
        \item If $\s\preceq_\alpha \tau$ then $o(\tau)\le o(\s)$; 
        \item If $\s\preceq_\alpha \tau$ and $f(\tau)\ne f(\s)$ then $o(\tau)< o(\s)$;
        \item If $o(\s)=\eta$ then $f(\s)=0$.  
    \end{orderedlist}
\end{proposition}

This notion of a ``witness'' for the convergence of an approximation is widely used for computable approximations of functions $F\colon \Nat\to \Nat$ (see for example \cite[Ch.2]{BIG}). The idea is that the well-foundedness of $\eta+1$ guarantees that the sequence $\seq{f(\s)}_{\s\prec_\alpha x}$ eventually stabilises, i.e., that ${f}$ indeed $\alpha$-converges to a limit~$F$. In general, the more ``mind-changes'' the approximation has, the more complicated the function~$F$ can be. The longer~$\eta$ is, the ``more room'' we have for mind-changes.

\begin{proof}
    In one direction, let~$A$ be a $D_\eta(\Sigma^0_{1+\alpha})$ set; let $\seq{A_{i}}_{i<\eta}$ be the sequence of uniformly $\Sigma^0_{1+\alpha}$ sets used to define~$A$. For uniformity of notation, let $A_\eta = \Baire$; Then $\seq{A_i}_{i\le \eta}$ is increasing and uniformly $\Sigma^0_{1+\alpha}$, and $x\in A$ if and only if parity$(i)\ne$ parity$(\eta)$ for the least~$i$ such that $x\in A_i$. 

    By \ref{TSP:Sigma_alpha_sets}, let $U_i$ be a sequence of uniformly c.e.\ subsets of~$\w^{<\w}$ such that $A_i = [U_i]^\prec_\alpha$. As in the proof of \cref{prop:alpha-limit_lemma}, we may assume that the sets~$U_i$ are uniformly computable, and each is upwards closed in $(\w^{<\w},\preceq_\alpha)$. We may take $U_\eta = \w^{<\w}$. 

    Recall that we are working with a computable copy of~$\eta$, i.e., a computable well-ordering $<^*$ of~$\Nat$ such that $(\Nat, <^*)\cong \eta$. Let $n\mapsto i_n$ be the isomorphism. For each $\s\in \w^{<\w}$, let $O(\s) = \{\eta\}\cup  \left\{ i_n \,:\,  n<|\s| \right\}$ (more thematically, we should take $n<|\s|_\alpha$, where as above $|\s|_\alpha$ is the height of~$\s$ in the tree $(\w^{<\w},\preceq_\alpha)$). So $O(\s)$ is a finite subset of $\eta+1$, the map $\s\mapsto O(\s)$ is computable, $O(\s)\subseteq O(\tau)$ if $\s\preceq_\alpha \tau$, $\bigcup_{\s\prec_\alpha x} O(\s) = \eta+1$ for all $x\in \Baire$, and $\eta\in O(\s)$ for all~$\s$.    

    Define $f\colon \w^{<\w}\to \{0,1\}$ and $o\colon \w^{<\w}\to \eta+1$ as follows. Given $\s\in \w^{<\w}$, let~$o(\s)$ be the least $i\in O(\s)$ (in the ordering of~$\eta$) such that $\s\in U_i$; let $f(\s)=1$ if parity$(o(\s))\ne $ parity$(\eta)$, $0$ otherwise. 

    It is not difficult to check that~$o$ and~$f$ are computable, and that the conditions (i)--(iii) hold. For each $x\in \Baire$, $1_A(x) = \lim_{\s\prec_\alpha x} f(\s)$ because for almost all $\s\prec_\alpha x$ we have $o(\s) = o(x) = \min \{i<\eta\,:\, x\in A_i\}$.

    \medskip
    
    In the other direction, let $f$ be an $\alpha$-computable approximation of the characteristic function $1_A$ of~$A$, with a witnessing function~$o$ as described. We would like to let $U_i = \left\{ \s\in \w^{<\w} \,:\,  o(\s)\le i  \right\}$ and $A_i = [U]^\prec_\alpha$. This would work if $f(\s)=1$ if and only if parity$(o(\s))\ne$ parity$(\eta)$. This, however, is only guaranteed when $o(\s)=\eta$ (this is condition~(iii)). So it remains to show that we can modify~$o$ to get a witness $\tilde o\colon \w^{<\w}\to \eta+1$ satisfying (i)---(iii) and also satisfying $f(\s)=1$ if and only if parity$(\tilde o(\s))\ne$ parity$(\eta)$. 

    This is easily done by setting $\tilde o(\s)$ to be either $o(\s)$ or $o(\s)+1$, the choice determined by $f(\s)$ and the parity of the ordinals. As mentioned, (iii) for~$o$ shows that if $o(\s) = \eta$ then $\tilde o(\s)=\eta$ as well; we never need to choose the value $\eta+1$. We need to verify~(i) for $\tilde o$ (and then~(ii) follows). Suppose that $\s\preceq_\alpha \tau$. Since $o(\tau)\le o(\s)$, we would only have a problem if $o(\tau) = o(\s) = i$ but $\tilde o(\s) = i+1$ while $\tilde o(\tau) = i$. But this implies that $f(\s)\ne f(\tau)$, which is impossible since~$o$ satisfies~(ii). 
\end{proof}

\begin{remark}\label{rmk:eta-c.a.}
    Ershov showed that a function $F\colon \Nat\to \Nat$ is $D_\eta(\Sigma^0_1)$-measurable if and only if it has a computable approximation~$f$ with a $<\eta$-witness: a computable function $o\colon \Nat^2\to \eta$ such that $o(n,s)\ge o(n,t)$ when $t\ge s$ and $o(n,s)>o(n,t)$ when in addition $f(n,s)\ne f(n,t)$. Such functions are called \emph{$\eta$-computably approximable} in \cite{BIG}.  In particular, a set $A\subseteq \Nat$ is in Ershov's ambiguous class $\Delta^{-1}_\eta$ (i.e., it is both $D_\eta(\Sigma^0_1)$ and co-$D_\eta(\Sigma^0_1)$) if and only if its characteristic function is $\eta$-c.a. The idea is that for each~$n$, since we know that $f(n)$ is defined, we can wait for an approximation of one the sets $F^{-1}\{m\}$ to give us an ordinal below~$\eta$, and then start our approximation from that point. 

    We can similarly define a notion of an $<\eta$-witness for a computable $\alpha$-appproximation of a $\Delta^0_{1+\alpha+1}$-measurable function $F\colon \Baire\to \Nat$: a function~$o$ as in \cref{prop:D_eta_level_and_approximations} but which takes values below~$\eta$; the value~$\eta$ is not allowed, and so condition~(iii) is removed. 

    We would then hope that every $D_\eta(\Sigma^0_1)$-measurable function $F\colon \Baire\to \Nat$ has a computable $\alpha$-approximation with such a witness~$o$. This is almost the case; since we have to wait until we see an ordinal drop, the function~$o$ will not be defined on all of $\w^{<\w}$ but rather on a computable subset of $\w^{<\w}$ which is dense and upwards-closed in $(\w^{<\w},\prec_\alpha)$. We study such functions in greater detail in \cite{companionpaper}. 
\end{remark}

We can now prove the effective Hausdorff-Kuratowski theorem.

\begin{proof}[Proof of \cref{thm:effectiveHausdorff-Kuratowski}]
    By \cref{prop:alpha-limit_lemma,prop:D_eta_level_and_approximations}, it suffices to show: if $f\colon \w^{<\w}\to \Nat$ is a computable $\alpha$-approximation of a function $F\colon \Baire\to \Nat$, then there is some computable well-ordering~$\eta$ and some computable witness function $o\colon \w^{<\w}\to \eta$ satisfying~(i) and~(ii) of \cref{prop:D_eta_level_and_approximations}.  By applying this to $F = 1_A$, it will follow that $A$ is $D_{\eta+1}(\Sigma^0_\alpha)$ (since $o$ never takes value $\eta$, property~(iii) is irrelevant).

    Let $T\subset \w^{<\w}$ consist of the empty sequence, together with the finite sequences~$\s$ for which $f(\s)\ne f(\s^-)$, where~$\s^-$ is the immediate predecessor of~$\s$ on the tree $(\w^{<\w},\preceq_\alpha)$. As a subset of that tree, $(T,\preceq_\alpha)$ is a tree as well. It is well-founded: if $\seq{\s_i}_{i<\w}$ is an infinite path in~$T$, then by \ref{TSP:extends_prec} and \ref{TSP:unique_path}, $x= \bigcup_i \s_i$ is an element of Baire space and $\s_i\prec_\alpha x$ for all~$i$; but then $\seq{f(\s)}_{\s\prec_\alpha x}$ does not stabilise to~$F(x)$, contrary to the assumption that~$f$ is an $\alpha$-approximation of~$F$. 

    Now let~$\eta$ be the Kleene-Brouwer ordering of~$T$ (also known as the Lusin-Sierpi\'nski ordering \cite{Brouwer,Kleene:1955a,LusinSierpinski:23}, see \cite[4A.4]{Moschovakis:2ndEd}). The identity function is a computable rank function $r\colon T\to \eta$. Extend~$r$ to a function $o\colon \w^{<\w}\to \eta$ by letting $o(\tau) = r(\s)$ where~$\s$ is the longest $\preceq_\alpha$-predecessor of~$\tau$ which is on~$T$. 
\end{proof}

%%%%%%%%%%%%%%%%%%%%%%%%%%%%%%%%%%%%%%%%%%%%%%%%%
\subsection{Wadge's Theorem}
\label{subsec:Wadge}
%%%%%%%%%%%%%%%%%%%%%%%%%%%%%%%%%%%%%%%%%%%%%%%%%

Let $\lambda$ be a countable limit ordinal. Wadge described how to construct all sets in the class $\bDelta^0_\lambda$ starting from the sets in $\bDelta^0_{<\lambda}$ and closing under the operation of taking \emph{separated unions}.

\begin{definition}\label{def:separated_sequence}
A countable sequence of sets $(A_n)_{n \in \omega}$ is \emph{$\bSigma^0_\xi$-separated} if there is a sequence $(U_n)_{n \in \omega}$ of pairwise disjoint $\bSigma^0_\xi$ sets with $A_n \subseteq U_n$ for all $n$.
\end{definition}

\begin{theorem}[Wadge]\label{thm:Wadge}
Let $\lambda$ be a countable limit ordinal. The class $\bDelta^0_\lambda$ is the smallest collection of sets which:
\begin{itemize}
\item Contains $\bSigma^0_\xi$, for all $\xi < \lambda$; and
\item Is closed under unions of $\bSigma^0_\xi$-separated sequences, for all $\xi < \lambda$.
\end{itemize}
\end{theorem}

\begin{remark} 
    Wadge stated his theorem in terms of \emph{partitioned unions}, in which the pairwise disjoint $\bSigma^0_\xi$ sets of \cref{def:separated_sequence} are required to form a partition of~$\Baire$. Every $\bSigma^0_\xi$-separated union is of course a $\bSigma^0_{\xi+1}$-partitioned union, so for Wadge's theorem we may use either notion. The difference between seperated and partitioned unions is important in the analysis of all Borel Wadge classes. 
\end{remark}

For our proof of \cref{thm:Wadge}, we require one last property of the true stages machinery. As observed above, this property is a special feature of the development of the machinery in \cite{GreenbergTuretsky:Pi11}. As with \ref{TSP:continuous}, it reflects the fact that for limit~$\lambda$, $x^{(\lambda)}\equiv_\Tur \bigoplus_{\alpha<\lambda}x^{(\alpha)}$. Recall the notation introduced in  the proof of \cref{prop:alpha-limit_lemma}: $|\s|_\alpha$ is the height of~$\s$ in the tree $(\w^{<\w},\preceq_\alpha)$. 

\begin{TSPlist}[start=9]
    \item \label{TSP:limit}
    Let $\lambda < \wock$ be computable and limit. There is a computable and increasing sequence $\seq{\lambda_k}$, cofinal in~$\lambda$, such that for all $\s\in \w^{<\w}$, letting $k = |\s|_\lambda$, then $|\s|_{\lambda_k} = k$, and for all $\tau\in \w^{\le \w}$, 
    \[
        \s\preceq_\lambda \tau \,\Iff\,  \s\preceq_{\lambda_k} \tau. 
    \]
\end{TSPlist}

\begin{proof}[Proof of \cref{thm:Wadge}]
Let $\+C$ be the smallest collection of sets with the above closure properties. An induction shows that $\+C\subseteq \bDelta^0_\lambda$, so we will show the converse: given an $A \in \bDelta^0_\lambda$, we will demonstrate that $A \in \+C$. By working relative to an appropriate oracle, we may assume that $\lambda < \wock$ and $A \in \Delta^0_\lambda$.

Since $\lambda = 1+\lambda$, \ref{TSP:Sigma_alpha_sets} gives us two c.e.\ sets $W_0$ and~$W_1$ such that $A = [W_1]^\prec_\lambda$ and $A^\complement = \Baire\setminus A = [W_0]^\prec_\lambda$. As in previous proofs, we may assume that~$W_0$ and~$W_1$ are disjoint and upwards closed in $(\w^{<\w},\preceq_\lambda)$. 

Let $T = \w^{<\w}\setminus (W_0 \cup W_1)$.  Then $(T,\preceq_\lambda)$  is a subtree of $(\w^{<\w},\preceq_\lambda)$.  Further, since $[W_0]^\prec_\lambda \cup [W_1]^\prec_\lambda = \Baire$, $T$ is well-founded (as in the proof of \cref{thm:effectiveHausdorff-Kuratowski}, this uses \ref{TSP:extends_prec} and \ref{TSP:unique_path}). Let $r\colon T\to \w_1$ be a rank function for~$T$. 

We modify~$r$ to obtain a function $s\colon \w^{<\w}\to \w_1$ as follows:
\[
s(\sigma) =
\begin{cases*}
    r(\sigma)+1, & if $\sigma \in T$; \\
    0, & if $\sigma \not \in T$.
\end{cases*}
\]

For each $\s\in \w^{<\w}$ let $A_\s = A\cap [\s]^\prec_\lambda$. By induction on $s(\s)$, we will show that $A_\s\in \+C$. The result will then follow since $A = A_{\emptystring}$. Let $\seq{\lambda_k}$ be the sequence given by \ref{TSP:limit}. 
%As above we write $|\s|_\lambda$ for the height of~$\s$ in the tree $(\w^{<\w},\preceq_\lambda)$. 

Suppose that $s(\s)=0$. If $\s\in W_0$ then $A_\s = \emptyset$; if $\s\in W_1$ then $A_\s = [\s]^\prec_\lambda$. In either case, by \ref{TSP:limit} and \ref{TSP:Sigma_alpha_sets}, $A_\s$ is $\bSigma^0_{1+\lambda_k}$ where $k = |\s|_\lambda$, and so is in $\+C$. 

Suppose that $s(\s) > 0$, and thus $\s\in T$. Let~$R$ be the collection of immediate $\prec_\lambda$-extensions of~$\sigma$.  By the inductive hypothesis, $A_\tau\in \+C$ for every $\tau \in R$. Let $k = |\s|_\lambda+1$, so $k = |\tau|_\lambda$ for all $\tau\in R$. By \ref{TSP:limit} (and \ref{TSP:Sigma_alpha_sets}), the sequence $\smallseq{[\tau]^\prec_\lambda}_{\tau \in R}$ shows that $\smallseq{A_\tau}_{\tau\in R}$ is $\bSigma^0_{1+\lambda_k}$-separated. Hence $A_\s = \bigcup_{\tau\in R} A_\tau$ is in~$\+C$ as well. 
\end{proof}

Note that we did not actually need the sets~$W_0$ and~$W_1$ to be c.e. In other words, it was enough that~$A$ is $(z,\lambda)$-clopen for some parameter~$z$; we didn't use the fact that it is effectively $(z,\lambda)$-clopen, i.e., $\Delta^0_\lambda(z)$. Nevertheless, the proof can be made entirely effective, and an effective version of Wadge's theorem is true. 

To state it, for a lightface class~$\Gamma$ and a computable ordinal~$\xi$, let $\SU{\xi}{\Gamma}$ be the collection of unions of all sequences of sets which are uniformly in~$\Gamma$ and uniformly $\Sigma^0_\xi$-separated. For a computable limit ordinal $\lambda$ define by induction on $\beta<\wock$ the lightface classes $\ISU{\beta}{\lambda}$ (iterated separated unions):
\begin{itemize}
    \item $\ISU{<0}{\lambda} = \Sigma^0_{<\lambda} = \bigcup_{\xi<\lambda}\Sigma^0_\xi$; 
    \item For $\beta>0$, $\ISU{<\beta}{\lambda} = \bigcup_{\gamma<\beta} \ISU{\gamma}{\lambda}$; 
    \item $\ISU{\beta}{\lambda} = \SU{<\lambda}{\ISU{<\beta}{\lambda}} = \bigcup_{\xi<\lambda}\SU{\xi}{\ISU{<\beta}{\lambda}}$.
\end{itemize}
Each class $\ISU{\beta}{\lambda}$ depends on a choice of computable copy of~$\beta$; given such a copy, we can give an effective enumeration of the class $\ISU{<\beta}{\lambda}$, so we can indeed speak of a sequence of sets being uniformly in this class. The effective version of Wadge's theorem is:

\begin{theorem} \label{thm:effective_Wadge}
    Let $\lambda <\wock$ be a limit ordinal. Then
    \[
        \Delta^0_\lambda = \ISU{<\wock}{\lambda}. 
    \]
\end{theorem}

The proof above effectivises; as in the proof of \cref{thm:effectiveHausdorff-Kuratowski}, we can make~$r$ a computable ranking function into a computable ordinal (by taking the Kleene-Brouwer ordering on~$T$). Then, by effective transfinite recursion on~$T$ (from the leaves to the root), we show that $A_\s\in \ISU{s(\s)}{\lambda}$, uniformly so.

%%%%%%%%%%%%%%%%%%%%%%%%%%%%%%%%%%%%%%%%%%%%%%%%%
%%%%%%%%%%%%%%%%%%%%%%%%%%%%%%%%%%%%%%%%%%%%%%%%%
\section{Louveau and Saint-Raymond's separation theorem}
\label{sec:Louveau-Saint-Raymond}
%%%%%%%%%%%%%%%%%%%%%%%%%%%%%%%%%%%%%%%%%%%%%%%%%
%%%%%%%%%%%%%%%%%%%%%%%%%%%%%%%%%%%%%%%%%%%%%%%%%

We begin by recalling Louveau and Saint-Raymond's separation theorem.

\begin{theorem}[Louveau \& Saint-Raymond \cite{LouveauSR:WH}]\label{thm:LSR:boldface}
Suppose that $A \in \bSigma^0_{1+\xi}$ and that~$B_0$ and~$B_1$ are disjoint~$\bSigma^1_1$ sets.  Then at least one of the following holds:
\begin{itemize}
\item There is a continuous function $f$ such that $A = f^{-1}(B_1)$ and $A^\complement = f^{-1}(B_0)$.
\item There is a $\*\Sigma^0_{1+\xi}$ set $U$ with $B_0 \subseteq U$ and $B_1 \cap U = \emptyset$.
\end{itemize}
\end{theorem}
In \cite{LouveauSR:Strength}, Louveau and Saint-Raymond extend this theorem by considering all  non-self-dual Borel Wadge class. Further, in that paper, the authors note that their proof is sufficiently effective to yield the following:

\begin{theorem}\label{thm:LSR}
Fix $\xi < \wock$.  Suppose that $A \in \Sigma^0_{1+\xi}$ and that~$B_0$ and~$B_1$ are disjoint~$\Sigma^1_1$ sets.  Then at least one of the following holds:
\begin{itemize}
\item There is a continuous function~$f$ such that $A = f^{-1}(B_1)$ and $A^\complement = f^{-1}(B_0)$.
\item There is a $\Sigma^0_{1+\xi}(\Delta^1_1)$ set~$U$ with $B_0 \subseteq U$ and $B_1 \cap U = \emptyset$.
\end{itemize}
\end{theorem}
(Recall the notation $\Gamma(\Delta^1_1)$ from~\Cref{sub:the_hausdorff_kuratowski_theorem}.)

As usual, \cref{thm:LSR} implies \cref{thm:LSR:boldface} by relativisation. By taking~$A$ to be a $\Sigma^0_{1+\xi}$-complete set (which is therefore not $\bPi^0_{1+\xi}$), \cref{thm:LSR} implies Louveau's celebrated separation theorem:

\begin{theorem}[\cite{Louveau:80:Separation}] \label{thm:Louveau_separation}
  Let  $\xi<\wock$. If $B_0$ and~$B_1$ are disjoint $\Sigma^1_1$ sets and are separated by some~$\bSigma^0_{1+\xi}$ set, then they have a $\Sigma^0_{1+\xi}(\Delta^1_1)$ separator. 
\end{theorem}

The proof of \cref{thm:LSR} in \cite{LouveauSR:Strength} relies on a technique of unravelling complicated games into simpler ones. In \cite{LouveauSR:WH}, before they introduce the technique, the authors present a simpler proof of \cref{thm:LSR} for $\xi = 1$ and $\xi=2$ in which a closed game is directly defined without the need for unravelling. The proof we give is a generalisation of this simpler technique to all countable ordinals~$\xi$.

\begin{proof}[Proof of \cref{thm:LSR}]
Fix computable trees $T_0$, $T_1$ such that for $i=0,1$, $B_i$ is the projection of $[T_i]$: $B_i = \{ y \,:\, (\exists z)\,\, (y,z) \in [T_i]\}$. By \ref{TSP:Sigma_alpha_sets}, fix a computable, upwards closed $W\subseteq \w^{<\w}$ such that $[W]^\prec_\xi = A$. 

We define a two player game, in which players~I and~II alternate turns.  On player~I's $i\tth$ turn, they play $x_i \in \omega$.  On player~II's $i\tth$ turn, they play $(y_i, z_i) \in \omega\times \omega$.  For our game, the set of runs in which player I wins will be an open set, so we will explain how to determine partway through the game if player I has already won.

Suppose that $n>0$, player~I has played $x_0, \dots, x_{n-1}$, while player~II has played $(y_1, z_1)$, $\dots$, $(y_{n}, z_{n})$ (it is convenient to begin the indexing for player~II at~1 rather than~0).  So it is player~I's $n\tth$ turn, but first we will determine if they have already won.  Let $\b{x} = (x_0, \dots, x_{n-1}) \in \omega^{<\omega}$.  Then $\b{x}$ has an opinion as to whether the real~$x$ which player I is building will be in $A$: the opinion is determined by whether $\b{x} \in W$ or not.  Whichever it is, player II is responsible for building a real~$y$ in~$B_0$ or~$B_1$, as appropriate, and also a~$z$ which witnesses~$y$'s membership.  However, we will only ask player II to make progress on constructing~$z$ at those stages which appear $\xi$-true at the current stage and which share this opinion about~$x$.  This is made precise in the following.

Define a set of indices $F$ as follows:
\begin{itemize}
    \item If $\b{x} \in W$, let $F = \{ i \in \{1,\dots, n\} : \b{x}\rest{i} \preceq_\xi \b{x} \andd \b{x}\rest{i} \in W\}$;
    \item If $\b{x} \notin W$, let $F = \{ i \in \{1,\dots, n\} : \b{x}\rest{i} \preceq_\xi \b{x} \}$.
\end{itemize}
Note that $n \in F$. Also note that since~$W$ is upwards closed in $(\w^{<\w}, \preceq_\xi)$, if $\b{x}\notin W$ then $\b{x}\rest{i}\notin W$ for all $i\in F$. 

Enumerate~$F$ as $F = \{ a_0 < \dots < a_{k-1}\}$, so $k \ge 1$.  Let $\b{y} = (y_1, \dots, y_k)$ and $\b{z} = (z_{a_0}, \dots, z_{a_{k-1}})$, observing the difference in the subscripts: for~$\b{z}$, we only collect the numbers played by player~II in response to initial segments of~$\b{x}$ which appear to be correct. 

We declare that player I wins if:
\begin{itemize}
    \item $\b{x} \in W$ and $(\b{y}, \b{z}) \not \in T_1$; or
    \item $\b{x} \not \in W$ and $(\b{y}, \b{z}) \not \in T_0$. 
\end{itemize}
Otherwise, the game continues.  If the game continues for~$\omega$ many turns, then player~II wins.

As this is an open/closed game, it is determined. 

\begin{lrclaim}
    If player~II has a winning strategy in the game, then there is a continuous function reducing $(A^\complement,A)$ to $(B_0,B_1)$.
\end{lrclaim}

\begin{proof}
Suppose that~$S$ is a winning strategy for player~II.  For $x \in \Baire$, let $y, z \in \Baire$ be the sequences generated by player II when player I plays $x$ and player II plays according to $S$.  We claim that $x \mapsto y$ is our desired continuous reduction.

If $x \in A$, let $a_0 < a_1 < \cdots$ enumerate those $a < \omega$ with $x\rest{a} \preceq_\xi x$ and $x\rest{a} \in W$.  Let $v = (z_{a_0}, z_{a_1}, \dots)$.  Then for each $k$, since player~II had not lost at the start of player I's $a_k\tth$ turn, $(y\rest{k}, v\rest{k}) \in T_1$. So $(y, v) \in [T_1]$, and thus $y \in B_1$ as desired.

If $x \not \in A$, let $a_0 < a_1 < \cdots$ enumerate those $a < \omega$ with $x\rest{a} \preceq_\xi x$; note that $x\rest{a} \not \in W$ for such~$a$.  The same argument shows that $y \in B_0$, as desired.
\end{proof}

Now suppose instead that player~I has a winning strategy~$S$ in the game. In the rest of the proof, we show how to use~$S$ to construct a $\Sigma^0_{1+\xi}(S)$ separator between~$B_0$ and~$B_1$. Since the game is open for player~I, there is a hyperarithmetic winning strategy~$S$, so the separator can be taken to be $\Sigma^0_{1+\xi}(\Delta^1_1)$ as required. 

\medskip

For $y \in \Baire$ and $\sigma \in \omega^{<\omega}$, let $S(y, \sigma)$ be the sequence $(x_0, \dots, x_n)$ which results from player I playing according to~$S$ and player~II playing $(y\rest{|\sigma|}, \sigma)$.  Note that $|S(y,\s)|= |\s|+1$. It will be notationally convenient to posit the existence of a string $\pi$ with $|\pi| = -1$ and $S(y, \pi) = \emptystring$.

Roughly, the idea is the following. Given~$y$, we need to decide on which side of the separator we put it (on the $B_0$ side or the~$B_1$ side). We put it on the~$B_0$ side if there is some ``credible''~$\s$ for which $S(y,\s)\in W$. We need to ensure that if there is such~$\s$ then $y\notin B_1$ (and similarly, if there is no such~$\s$, then $y\notin B_0$). The argument for contradiction will be that if $y\in B_1$, say $(y,z)\in [T_1]$, then we can start from~$\s$ and use~$z$ to defeat~$S$. To do this, we will define a sequence $\s_0,\s_1,\dots$ such that for each~$i$, after player~I's response to our play of $(y,\s_i)$, we can play~$z_i$ and continue to the next~$\s_{i+1}$. The notion of credibility of a sequence, which also allows us to continue the inductive construction, is that of a ``strongly $\xi$-correct'' sequence (\cref{def:alpha-correct}). The complexity of the separator depends on the complexity of this notion of correctness, which is computed in \cref{clm:correctness_complexity}. The heart of the argument is \cref{claim:main:extending_correct_sequences}, which ensures that we can take another step in the construction we outlined. 

\medskip

We will use the strategy~$S$ to pull back the true stage relations onto strings~$\sigma$, which we think of as potential second coordinates for player II to play along a given real~$y$. 
\begin{lrdef}
Let $\alpha \le \xi$ and $y \in \Baire$. For $\sigma, \tau \in \omega^{<\omega}\cup \{\pi\}$, we write $\sigma \tleq_\alpha^y \tau$ when $\sigma \preceq \tau$ and $S(y, \sigma) \preceq_\alpha S(y, \tau)$.
\end{lrdef}

We will now define a sort of absolutely true stage for these new relations, moderated by the need to avoid letting player I win.

\begin{lrdef}\label{def:alpha-correct}
For $y \in \omega^\omega$, $\sigma \in \omega^{<\omega} \cup \{\pi\}$ and $\alpha \le \xi$, we define what it means for $\sigma$ to be \emph{$\alpha$-correct} or \emph{strongly $\alpha$-correct} for~$y$.
\begin{itemize}
\item $\sigma$ is 0-correct for~$y$ if in the partial play of the game where player II plays $(y\rest{|\sigma|}, \sigma)$ and player I plays according to $S$, player I has not yet won.
\item $\sigma$ is $(\alpha+1)$-correct for $y$ if it is strongly $\alpha$-correct for $y$, and for every $\tau \treq{y}{\alpha} \sigma$ which is strongly $\alpha$-correct for $y$, $\tau \treq{y}{\alpha+1}  \sigma$.
\item For $\lambda$ a limit, $\sigma$ is $\lambda$-correct for $y$ if it is $\beta$-correct for $y$ for all $\beta < \lambda$.
\item $\sigma$ is strongly $\alpha$-correct for $y$ if for every $\tau \tleq_\alpha^y \sigma$, $\tau$ is $\alpha$-correct for $y$.
\end{itemize}
\end{lrdef}

\begin{lrclaim} \label{clm:correctness_basics}
Let $y\in \Baire$, $\alpha\le \xi$, and $\s\in \w^{<\w}\cup \{\pi\}$. 
\begin{sublemma}
    \item \label{item:correctness_basics:strong_implies_weak} 
    If~$\sigma$ is strongly $\alpha$-correct for $y$, then it is $\alpha$-correct for $y$.
    
    \item \label{item:correctness_basics:strong-0-correct}
     $\s$ is strongly 0-correct for~$y$ if and only if it is 0-correct for~$y$. 

    \item \label{item:correctness_basics:smaller_beta}
     If $\sigma$ is $\alpha$-correct for $y$, then it is strongly $\beta$-correct for $y$ for every $\beta < \alpha$.

    \item \label{item:correctness_basics:initial_segment}
     If $\sigma$ is strongly $\alpha$-correct for $y$, then every $\rho \tleq_\alpha^y \sigma$ is strongly $\alpha$-correct for~$y$.

    \item \label{item:correctness_basics:initial_segment2}
    If $\s\preceq \tau$ and both~$\s$ and~$\tau$ are $\alpha$-correct for~$y$, then $\s\tleq^y_\alpha \tau$. 

    \item \label{item:correctness_basics:root}
     $\pi$ is strongly $\alpha$-correct for~$y$.
\end{sublemma}
\end{lrclaim}

\begin{proof}
    \ref{item:correctness_basics:strong_implies_weak} follows from the relations $\preceq_\alpha$ (and thus $\tleq^y_\alpha$) being reflexive. \ref{item:correctness_basics:strong-0-correct} holds because once player~I wins, the game ends. \ref{item:correctness_basics:smaller_beta} is by induction on~$\alpha$.  \ref{item:correctness_basics:initial_segment} follows from the transitivity of $\le_\alpha$ (and thus of $\tleq_\alpha^y$). \ref{item:correctness_basics:initial_segment2} is proved by induction on~$\alpha$. 
     \ref{item:correctness_basics:root} follows from $\emptystring\preceq_\alpha \s$ for all~$\s$ (\ref{TSP:tree}). 
\end{proof}

Note that \ref{item:correctness_basics:initial_segment} and \ref{item:correctness_basics:initial_segment2} of \cref{clm:correctness_basics} together imply that if $\sigma$ is strongly $\alpha$-correct for $y$, then for all $\rho\preceq \s$, $\rho \tleq_\alpha^y \sigma$ if and only if $\rho$ is strongly $\alpha$-correct for~$y$ if and only if $\rho$ is $\alpha$-correct for~$y$.

\begin{lrclaim} \label{clm:correctness_complexity}
Let $\alpha\le \xi$. The relations ``$\s$ is $\alpha$-correct for~$y$'' and ``$\s$ is strongly $\alpha$-correct for~$y$'' are:
\begin{orderedlist}
\item $\Delta^0_1(S)$ if $\alpha=0$;
\item $\Pi^0_\alpha(S)$ if $0<\alpha<\w$;
\item $\Delta^0_\alpha(S)$ for limit $\alpha$; and
\item $\Pi^0_{\alpha-1}(S)$ if $\alpha>\w$ is a successor. 
\end{orderedlist}
% These are uniform in~$\alpha$.
\end{lrclaim}

\begin{proof}
We prove this by induction on~$\alpha$. Technically, of course, this is an instance of effective transfinite recursion: for the limit case, we need the relations to be uniformly in their classes. 
For every~$\alpha$, the complexity of the strong relation follows from the complexity of  ``$\s$ is $\alpha$-correct for~$y$'' by the fact that the relations~$\le_\alpha$ are uniformly computable (\ref{TSP:computable}). 

All cases are immediate, except for the limit case~(iii). Suppose that $\lambda \le \xi$ is a limit ordinal. We use \ref{TSP:limit}; let $\seq{\lambda_k}$ be given by that property. 

Given~$y$ and~$\sigma$, let $k = |S(y,\s)|_\lambda$. 
% \[
%     k = \max \left\{ |S(y,\rho)|_\lambda \,:\, \rho\preceq \s  \right\}. 
% \]
We claim that~$\s$ is $\lambda$-correct for~$y$ if and only if it is strongly $\lambda_k$-correct for~$y$. One direction follows from \cref{clm:correctness_basics}\ref{item:correctness_basics:smaller_beta}. For the other direction, suppose that~$\s$ is strongly $\lambda_k$-correct for~$y$. By induction on $\beta\in [\lambda_k,\lambda]$, we can show that every $\rho\tleq^y_{\lambda_k} \s$ is strongly $\beta$-correct for~$y$. 

This is mostly chasing the definitions, using \cref{clm:correctness_basics}. The main point is that for all $\rho\tleq^y_{\lambda_k} \s$ we have $|S(y,\rho)|_{\lambda_k} = |S(y,\rho)|_\lambda \le k$, and so for all~$\tau$, if $\rho\tleq^y_{\beta} \tau$ then $\rho \tleq^y_\lambda \tau$, and so $\rho \tleq^y_{\beta+1} \tau$.    
\end{proof}

\begin{lrclaim} \label{claim:main:extending_correct_sequences}
Let $y \in \Baire$ and $\alpha\le \xi$. Suppose that $\rho \in \omega^{<\omega}\cup \{\pi\}$ is strongly $\alpha$-correct for~$y$.  If~$\s$ is a one-element extension of~$\rho$ which is 0-correct for~$y$, then there exists an $\tau \succeq \s$  which is strongly $\alpha$-correct for $y$.
\end{lrclaim}

For the purposes of this claim, the empty sequence is a one-element extension of~$\pi$. 

\begin{proof}
The argument is by induction on $\alpha$. For $\alpha=0$ we can take $\tau=\s$ (\cref{clm:correctness_basics}\ref{item:correctness_basics:strong-0-correct}).

\medskip

For the successor case, we use \ref{TSP:p_function}. Suppose that the lemma holds for $\alpha<\xi$, and suppose that~$\rho$ is strongly $(\alpha+1)$-correct for~$y$. Let~$p_\alpha$ be given by \ref{TSP:p_function}.  For brevity, we will write $p(\eta)$ in place of $p_\alpha(S(y, \eta))$. By induction, there are $\tau \succeq \s$ which are strongly $\alpha$-correct for $y$.  Amongst those~$\tau$, choose one to minimize $p(\tau)$.  We claim that this~$\tau$ is strongly $(\alpha+1)$-correct for~$y$.  Note that since~$\rho$ is $(\alpha+1)$-correct for~$y$ and~$\tau$ is strongly~$\alpha$-correct for~$y$, $\rho\tleq^y_\alpha \tau$ (by \cref{clm:correctness_basics}\ref{item:correctness_basics:initial_segment2}) and so $\rho \tleq_{\alpha+1}^y \tau$ (by \cref{def:alpha-correct}). 

We must show that every $\nu \tleq_{\alpha+1}^y \tau$ is $(\alpha+1)$-correct for~$y$. If $\nu \prec \s$, then $\nu \preceq \rho$; since $\preceq_{\alpha+1}$ is a tree (\ref{TSP:tree}), $\nu \tleq_{\alpha+1}^y \rho$.  As~$\rho$ is strongly $(\alpha+1)$-correct for $y$, $\nu$ is $(\alpha+1)$-correct for~$y$.  So we may assume that $\s \preceq \nu$.

We need to show that for every $\eta \treq{y}{\alpha} \nu$ which is strongly $\alpha$-correct for~$y$, it is the case that $\nu \tleq_{\alpha+1}^y \eta$.  By \ref{TSP:p_function}, since $\nu\tleq^{y}_{\alpha+1} \tau$, $p(\nu) \le p(\tau)$.  As~$\eta$ is strongly $\alpha$-correct for $y$, for every $\eta'$ with $\nu \tleq_\alpha^y \eta' \tleq_\alpha^y \eta$, $\eta'$ is strongly $\alpha$-correct for~$y$ (\cref{clm:correctness_basics}\ref{item:correctness_basics:initial_segment}).  Thus such~$\eta'$ was a candidate for~$\tau$, so $p(\tau) \le p(\eta')$, and thus $p(\nu) \le p(\eta')$.  By \ref{TSP:p_function} again, $\nu \tleq_{\alpha+1}^y \eta$.  %Thus~$\eta$ is strongly $(\alpha+1)$-correct for $y$.

\medskip

For $\alpha$ a limit, we again use \ref{TSP:limit}. Let $\seq{\alpha_k}$ be a cofinal sequence in~$\alpha$ supplied by that property; let $k = |S(y,\rho)|_\alpha+1$. By induction, there is some $\tau\succeq \s$ which is strongly $\alpha_k$-correct for~$y$. Among such~$\tau$, fix one minimal with respect to~$\tleq^y_{\alpha_k}$. So $\nu\tle^y_{\alpha_k} \tau$ implies $\nu\tleq^y_{\alpha_k} \rho$. Since $|S(y,\nu)|_{\alpha_k} = |S(y,\nu)|_\alpha$ for all $\nu \tleq^y_{\alpha_k} \rho$, we conclude that $|S(y,\tau)|_\alpha = k$.
% For if $\bar{x}\prec S(y,\eta)$ has length $n+1$ then $\bar{x} = S(y,\eta\rest{n})$; if $\bar{x}\prec_{\alpha} S(y,\eta)$, i.e., if $\eta\rest{n}\tle^y_{\alpha} \eta$, then $\bar{x} \preceq_{\alpha} S(y,\s)$. 
 Hence, for all~$\eta$, if $\tau\tleq^y_{\alpha_k} \eta$ then $\tau\tleq^y_\alpha \eta$. As~$\rho$ is strongly $\alpha$-correct for~$y$, this suffices to show, by induction, on $\beta\in [\alpha_k,\alpha]$, that~$\tau$ is strongly~$\beta$-correct for~$y$.
\end{proof}

We are now ready to define a separator~$U$.
\[
U = \{ y : \text{ there is some $\s\in \w^{<\w}$, strongly $\xi$-correct for~$y$, with } S(y, \sigma) \in W\}.
\]
By~\cref{clm:correctness_complexity}, the set~$U$ is $\Sigma^0_{1+\xi}(S)$. We need to show that $B_0\subseteq U$ and $B_1\subseteq U^\complement$. 

Suppose, for a contradition, that there is some $y\in U\cap B_1$. Fix $v = (v_0, v_1, \dots) \in \Baire$ such that $(y, v) \in [T_1]$. We define a sequence $\s_0\prec\s_1 \prec\cdots$ of sequences which together with~$y$ will give a play for~II which defeats the strategy~$S$. We ensure that:
\begin{orderedlist}
    \item Each~$\s_i$ is strongly $\xi$-correct for~$y$; 
    \item $\s_{i+1}(|\s_i|) = v_i$; and
    \item for each~$i$, 
    \[
        \{\rho \in \w^{<\w}  \,:\,  \rho \tleq^y_\xi \s_i \andd S(y, \rho) \in W\} = \{ \sigma_j \,:\, j \le i\}.
    \]
\end{orderedlist}
For~$\s_0$, we choose some sequence witnessing that $y\in U$, minimal such; so for all $\rho\tle^y_\xi \s_0$ in~$\w^{<\w}$ we have $S(y,\rho)\notin W$. Note that $\s_0\ne \pi$. Suppose that~$\s_i$ has been chosen. We then consider $\tau  = \s_i\conc v_i$; we wish to show that~$\tau$ is 0-correct for~$y$, i.e., that player~I does not win in response to $(y\rest{|\tau|},\tau)$. Since~$\s_i$ is $\xi$-correct for~$y$, it is 0-correct for~$y$; player~I has not won after we play $(y\rest{|\s_i|},\s_i)$. After playing $(y\rest{|\tau|},\tau)$, by~(iii) above, the set~$F$ of indices used to assess whether~I wins is $\{|\s_j|+1\,:\, j \le i\}$ and so the corresponding $\b{z}$ is $v\rest{(i+1)}$. Since $(y\rest{(i+1)},\b{z})\in T_1$, $\tau$ is 0-correct for~$y$ as required.  

We can therefore appeal to \cref{claim:main:extending_correct_sequences}. We choose~$\s_{i+1}$ as given by the claim, of minimal length.  This ensures~(iii) for~$i+1$ (we use the fact that~$W$ is upwards closed in $(\w^{<\w},\preceq_\xi)$). 

\smallskip

The other case, that $y\in B_0\setminus U$, is almost identical. In this case we will have 
 \[
        \{\rho \in \w^{<\w} \,:\,  \rho \tleq^y_\xi \s_i \} = \{ \sigma_j \,:\, j \le i\}.
    \]
The empty sequence is 0-correct for~$y$, and so by \cref{claim:main:extending_correct_sequences} and \cref{clm:correctness_basics}\ref{item:correctness_basics:root}, there is some $\s\in \w^{<\w}$ which is strongly $\xi$-correct for~$y$. We start with~$\s_0$ being such a string of minimal length (again $\s_0\ne \pi$). 
\end{proof}

As mentioned, Theorem~\ref{thm:LSR} holds not just for the classes $\Sigma^0_{1+\xi}$, but for any non-self-dual Borel Wadge class, and our methods can be used to show this.  From that, as in \cite{LouveauSR:Strength}, we can obtain Louveau separation for all non-self-dual Borel Wadge classes.  The requirement $\xi < \wock$ is replaced by the requirement that the Wadge class must have a~$\Delta^1_1$ name.  Of course, some care must be taken in defining what a name for a Wadge class is, and what the corresponding lightface class is.  See the companion paper~\cite{companionpaper} for further details.

%%%%%%%%%%%%%%%%%%%%%%%%%%%%%%%%%%%%%%%%%%%%%%%%%
\subsection{Reverse mathematics} % (fold)
\label{sub:reverse_mathematics}
%%%%%%%%%%%%%%%%%%%%%%%%%%%%%%%%%%%%%%%%%%%%%%%%%

As mentioned in the introduction, the proofs we gave are sufficiently effective so that they can be made within the system $\ATR_0$ of reverse mathematics. This is in some sense the weakest system which allows a meaningful development of Borel sets (\cite[Ch.V]{Simpson}). The main point is that the true stages machinery is effective, and so can be defined in $\ATR_0$ and all of its properties are provable in~$\ATR_0$. Thus:

\begin{theorem} \label{thm:ATR}
    Louveau's separation theorem (\cref{thm:Louveau_separation}) is provable in $\ATR_0$. 
\end{theorem}

We remark that in contrast, Louveau's original proof in \cite{Louveau:80:Separation}, which is the proof which appears in standard texts such as \cite{SacksBook} or \cite{Arnie_Miller:book}, appears to require the stronger system $\PiCA_0$.

\bibliography{Louveau_Sep_Bulletin}

\begin{thebibliography}{DGHTT}

\bibitem[Add59]{Addison:59:Separation}
J.~W. Addison.
\newblock Separation principles in the hierarchies of classical and effective
  descriptive set theory.
\newblock {\em Fund. Math.}, 46:123--135, 1959.

\bibitem[AK00]{AshKnight:book}
Chris~J. Ash and Julia~F. Knight.
\newblock {\em Computable structures and the hyperarithmetical hierarchy},
  volume 144 of {\em Studies in Logic and the Foundations of Mathematics}.
\newblock North-Holland Publishing Co., Amsterdam, 2000.

\bibitem[Ash86]{Ash:86:meta}
C.~J. Ash.
\newblock Recursive labelling systems and stability of recursive structures in
  hyperarithmetical degrees.
\newblock {\em Trans. Amer. Math. Soc.}, 298(2):497--514, 1986.

\bibitem[Ash90]{Ash:90:Watnick}
C.~J. Ash.
\newblock Labelling systems and r.e. structures.
\newblock {\em Ann. Pure Appl. Logic}, 47(2):99--119, 1990.

\bibitem[Bro24]{Brouwer}
Luitzen~E.J. Brouwer.
\newblock Beweiss, dass jede volle function gleichm\"{a}ssig stetig ist.
\newblock {\em Konin. Neder. Akad. van Weten. te Amst. Proc.}, 27:189--193,
  1924.
\newblock Reprinted in {L}.{E}.{J}.~{B}rouwer, {C}ollected {W}orks, vol.~1,
  {N}orth-{H}olland, {A}msterdam, 1975.

\bibitem[DDW]{DayDowneyWestrick}
Adam~R. Day, Rod~G. Downey, and Linda~B. Westrick.
\newblock Three topological reducibilities for discontinuous functions.
\newblock To appear.

\bibitem[Deb87]{Debs:87}
Gabriel Debs.
\newblock Effective properties in compact sets of {B}orel functions.
\newblock {\em Mathematika}, 34(1):64--68, 1987.

\bibitem[Dek54]{Dekker:Hypersimple}
J.~C.~E. Dekker.
\newblock A theorem on hypersimple sets.
\newblock {\em Proc. Amer. Math. Soc.}, 5:791--796, 1954.

\bibitem[DG20]{BIG}
Rod Downey and Noam Greenberg.
\newblock {\em A hierarchy of {T}uring degrees}, volume 206 of {\em Annals of
  Mathematics Studies}.
\newblock Princeton University Press, Princeton, NJ, 2020.
\newblock A transfinite hierarchy of lowness notions in the computably
  enumerable degrees, unifying classes, and natural definability.

\bibitem[DGHTT]{companionpaper}
Adam~R. Day, Noam Greenberg, Matthew Harrison-Trainor, and Daniel Turetsky.
\newblock An effective classification of {B}orel {W}adge classes.
\newblock In preparation.

\bibitem[DM]{DayMarks}
Adam~R. Day and Andrew~S. Marks.
\newblock The decomposability conjecture.
\newblock In preparation.

\bibitem[Ers68]{Ershov2}
Yuri~L. Ershov.
\newblock A certain hierarchy of sets. {II}.
\newblock {\em Algebra i Logika}, 7(4):15--47, 1968.

\bibitem[Fri57]{Friedberg:Priority}
Richard~M. Friedberg.
\newblock Two recursively enumerable sets of incomparable degrees of
  unsolvability (solution of {P}ost's problem, 1944).
\newblock {\em Proc. Nat. Acad. Sci. U.S.A.}, 43:236--238, 1957.

\bibitem[Fri71]{FriedmanH:BorelDeterminacy}
Harvey~M. Friedman.
\newblock Higher set theory and mathematical practice.
\newblock {\em Ann. Math. Logic}, 2(3):325--357, 1970/71.

\bibitem[GT]{GreenbergTuretsky:Pi11}
Noam Greenberg and Daniel Turetsky.
\newblock Completeness of the hyperarithmetic isomorphism relation.
\newblock To appear.

\bibitem[Har78]{Harrington:AnalyticDeterminacy}
Leo Harrington.
\newblock Analytic determinacy and {$0^{\sharp }$}.
\newblock {\em J. Symbolic Logic}, 43(4):685--693, 1978.

\bibitem[Hau49]{Haudorff:Grundzige:49}
Felix Hausdorff.
\newblock {\em Grundz\"{u}ge der {M}engenlehre}.
\newblock Chelsea Publishing Co., New York, N. Y., 1949.
\newblock Originally published: {V}eit \& {C}omp., {L}eipzig, 1914.

\bibitem[Hjo96]{Hjorth:Wadge:Pi12}
Greg Hjorth.
\newblock {$\bPi^1_2$} {W}adge degrees.
\newblock {\em Ann. Pure Appl. Logic}, 77(1):53--74, 1996.

\bibitem[HKL90]{HarringtonKechrisLouveau}
L.~A. Harrington, A.~S. Kechris, and A.~Louveau.
\newblock A {G}limm-{E}ffros dichotomy for {B}orel equivalence relations.
\newblock {\em J. Amer. Math. Soc.}, 3(4):903--928, 1990.

\bibitem[Kec95]{Kechris:book}
Alexander~S. Kechris.
\newblock {\em Classical descriptive set theory}, volume 156 of {\em Graduate
  Texts in Mathematics}.
\newblock Springer-Verlag, New York, 1995.

\bibitem[Kle43]{Kleene:1943}
S.~C. Kleene.
\newblock Recursive predicates and quantifiers.
\newblock {\em Trans. Amer. Math. Soc.}, 53:41--73, 1943.

\bibitem[Kle55a]{Kleene:55:Souslin}
S.~C. Kleene.
\newblock Hierarchies of number-theoretic predicates.
\newblock {\em Bull. Amer. Math. Soc.}, 61:193--213, 1955.

\bibitem[Kle55b]{Kleene:1955a}
S.~C. Kleene.
\newblock On the forms of the predicates in the theory of constructive
  ordinals. {II}.
\newblock {\em Amer. J. Math.}, 77:405--428, 1955.

\bibitem[Kni90a]{Knight:Workers:Finite}
J.~F. Knight.
\newblock A metatheorem for constructions by finitely many workers.
\newblock {\em J. Symbolic Logic}, 55(2):787--804, 1990.

\bibitem[Kni90b]{Knight:Worker:Transfinite}
Julia~F. Knight.
\newblock Constructions by transfinitely many workers.
\newblock {\em Ann. Pure Appl. Logic}, 48(3):237--259, 1990.

\bibitem[Kon38]{kondo:38}
Motokiti Kondo.
\newblock Sur l'uniformisation des compl\'ementaires analytiques et les
  ensembles projectifs de la seconde classe.
\newblock {\em Japan Journal of Mathematics}, 15:197--230, 1938.

\bibitem[KST99]{KechrisSoleckiTodorcevic}
Alexander~S. Kechris, S{\l}awomir~J. Solecki, and Stevo Todor\v{c}evi\'{c}.
\newblock Borel chromatic numbers.
\newblock {\em Adv. Math.}, 141(1):1--44, 1999.

\bibitem[Kur33]{Kuratowski:33}
Casimir Kuratowski.
\newblock Sur le prolongement de l'hom{\'e}omorphie.
\newblock {\em Comptes Rendus de l'Acad{\'e}mie des Sciences Paris},
  197:1090--1091, 1933.

\bibitem[Kur66]{Kuratowksi:Topology}
Casimir Kuratowski.
\newblock {\em Topology. {V}ol. {I}}.
\newblock Academic Press, New York-London; Pa\'{n}stwowe Wydawnictwo Naukowe
  [Polish Scientific Publishers], Warsaw, 1966.
\newblock Originally published as {T}opologie {I}, {M}onografie {M}atematyczne,
  vol.~{III}, {W}arszawa 1933.

\bibitem[Ler10]{Lerman:Framework:Book}
Manuel Lerman.
\newblock {\em A framework for priority arguments}, volume~34 of {\em Lecture
  Notes in Logic}.
\newblock Association for Symbolic Logic, La Jolla, CA; Cambridge University
  Press, Cambridge, 2010.

\bibitem[LL97]{LemppLerman:abstract}
Steffen Lempp and Manuel Lerman.
\newblock Iterated trees of strategies and priority arguments.
\newblock volume~36, pages 297--312. 1997.
\newblock Sacks Symposium (Cambridge, MA, 1993).

\bibitem[Lou80]{Louveau:80:Separation}
Alain Louveau.
\newblock A separation theorem for {$\Sigma ^{1}_{1}$} sets.
\newblock {\em Trans. Amer. Math. Soc.}, 260(2):363--378, 1980.

\bibitem[LS23]{LusinSierpinski:23}
Nikolai~Nikolaevich Lusin and Wac{\l}aw Sierpi\'{n}ski.
\newblock Sur un ensemble non measurable {B}.
\newblock {\em Journal de Math{\'e}matiques {$9^{e}$} serie}, 2:53---72, 1923.

\bibitem[LSR87]{LouveauSR:WH}
A.~Louveau and J.~Saint-Raymond.
\newblock Borel classes and closed games: {W}adge-type and {H}urewicz-type
  results.
\newblock {\em Trans. Amer. Math. Soc.}, 304(2):431--467, 1987.

\bibitem[LSR88]{LouveauSR:Strength}
Alain Louveau and Jean Saint-Raymond.
\newblock The strength of {B}orel {W}adge determinacy.
\newblock In {\em Cabal {S}eminar 81--85}, volume 1333 of {\em Lecture Notes in
  Math.}, pages 1--30. Springer, Berlin, 1988.

\bibitem[Lus17]{lusin:17}
Nikolai~Nikolaevich Lusin.
\newblock Sur la classification de {M}.~{B}aire.
\newblock {\em Comptes Rendus de l'Acad\'emie des Sciences Paris}, 164:91--94,
  1917.

\bibitem[LZ14]{LecomteZeleny}
Dominique Lecomte and Miroslav Zeleny.
\newblock Baire-class {$\xi$} colorings: the first three levels.
\newblock {\em Trans. Amer. Math. Soc.}, 366(5):2345--2373, 2014.

\bibitem[Mar75]{Martin:BorelDeterminacy}
Donald~A. Martin.
\newblock Borel determinacy.
\newblock {\em Ann. of Math. (2)}, 102(2):363--371, 1975.

\bibitem[Mil95]{Arnie_Miller:book}
Arnold~W. Miller.
\newblock {\em Descriptive set theory and forcing}, volume~4 of {\em Lecture
  Notes in Logic}.
\newblock Springer-Verlag, Berlin, 1995.
\newblock How to prove theorems about Borel sets the hard way.

\bibitem[MM11]{MarconeMontalban:Veblen}
Alberto Marcone and Antonio Montalb\'{a}n.
\newblock The {V}eblen functions for computability theorists.
\newblock {\em J. Symbolic Logic}, 76(2):575--602, 2011.

\bibitem[Mon14]{Montalban:TrueStages:paper}
Antonio Montalb\'{a}n.
\newblock Priority arguments via true stages.
\newblock {\em J. Symb. Log.}, 79(4):1315--1335, 2014.

\bibitem[Mos47]{Mostowski:1946}
Andrzej Mostowski.
\newblock On definable sets of positive integers.
\newblock {\em Fund. Math.}, 34:81--112, 1947.

\bibitem[Mos09]{Moschovakis:2ndEd}
Yiannis~N. Moschovakis.
\newblock {\em Descriptive set theory}, volume 155 of {\em Mathematical Surveys
  and Monographs}.
\newblock American Mathematical Society, Providence, RI, second edition, 2009.

\bibitem[Mu{\v{c}}56]{Muchnik:56}
A.~A. Mu{\v{c}}nik.
\newblock On the unsolvability of the problem of reducibility in the theory of
  algorithms.
\newblock {\em Dokl. Akad. Nauk SSSR (N.S.)}, 108:194--197, 1956.

\bibitem[Pau15]{Pauly}
Arno Pauly.
\newblock Computability on the countable ordinals and the
  {H}ausdorff-{K}uratowski theorem (extended abstract).
\newblock In {\em Mathematical foundations of computer science 2015. {P}art
  {I}}, volume 9234 of {\em Lecture Notes in Comput. Sci.}, pages 407--418.
  Springer, Heidelberg, 2015.

\bibitem[Pos48]{Post:48:report}
Emil~L. Post.
\newblock Degrees of recursive unsolvability: preliminary report (abstract).
\newblock {\em Bull. Amer. Math. Soc}, 54:641--642, 1948.

\bibitem[Sac90]{SacksBook}
Gerald~E. Sacks.
\newblock {\em Higher recursion theory}.
\newblock Perspectives in Mathematical Logic. Springer-Verlag, Berlin, 1990.

\bibitem[Sel03]{Selivanov:2003}
Victor Selivanov.
\newblock Wadge degrees of {$\omega$}-languages of deterministic {T}uring
  machines.
\newblock {\em Theor. Inform. Appl.}, 37(1):67--83, 2003.

\bibitem[Sho59]{Shoenfield:LimitLemma}
J.~R. Shoenfield.
\newblock On degrees of unsolvability.
\newblock {\em Ann. of Math. (2)}, 69:644--653, 1959.

\bibitem[Sie33]{Sierpinski:33}
Wac{\l}aw Sierpi\'{n}ski.
\newblock Sur une propri{\'e}t{\'e} des ensembles ${G}_\delta$ non
  d{\'e}nombrables.
\newblock {\em Polska Akademia Nauk. Fundamenta Mathematicae}, 21:66--72, 1933.

\bibitem[Sim99]{Simpson}
Stephen~G. Simpson.
\newblock {\em Subsystems of second order arithmetic}.
\newblock Perspectives in Mathematical Logic. Springer-Verlag, Berlin, 1999.

\bibitem[Sol78]{Solovay:78:HyperarithmeticallyEncoded}
Robert~M. Solovay.
\newblock Hyperarithmetically encodable sets.
\newblock {\em Trans. Amer. Math. Soc.}, 239:99--122, 1978.

\bibitem[Spe55]{Spector:Uniqueness}
Clifford Spector.
\newblock Recursive well-orderings.
\newblock {\em J. Symb. Logic}, 20:151--163, 1955.

\bibitem[Sus17]{souslin:17}
Mikhail~Yakovlevich Suslin.
\newblock Sur une definition des ensembles measurables {B} sans nombres
  transfinis.
\newblock {\em Comptes Rendus de l'Acad\'emie des Sciences Paris}, 164:88--91,
  1917.

\bibitem[Wad84]{Wadge:phd}
William~W. Wadge.
\newblock {\em Reducibility and determinateness on the Baire space}.
\newblock PhD thesis, University of California, Berkeley, 1984.

\end{thebibliography}
\bibliographystyle{alpha}

% \bibitem[Hau49]{Haudorff:Grundzige:49}
% Felix Hausdorff.
% \newblock {\em Grundz\"{u}ge der {M}engenlehre}.
% \newblock Chelsea Publishing Co., New York, N. Y., 1949.
% \newblock Originally published: Veit \& Comp., Leipzig, 1914.

% \bibitem[Kur66]{Kuratowksi:Topology}
% Casmir Kuratowski.
% \newblock {\em Topology. {V}ol. {I}}.
% \newblock Academic Press, New York-London; Pa\'{n}stwowe Wydawnictwo Naukowe
%   [Polish Scientific Publishers], Warsaw, 1966.
% \newblock New edition, revised and augmented, Translated from the French by J.
%   Jaworowski.
% \newblock Originally published as Topologie I, Monografie Matematyczne, vol.~III, Warszawa 1933.

\end{document}